\documentclass[11pt,twoside,reqno]{amsart}
\usepackage{amsmath}
\usepackage{hyperref}
\usepackage{amsxtra}
\usepackage[dvipsnames, x11names]{xcolor}
\usepackage{amsopn}
\usepackage{caption}
\usepackage[T1]{fontenc}
\usepackage[protrusion=true,expansion=true]{microtype}
\usepackage[pdftex]{graphicx}
\usepackage{url}
\usepackage{tabularray}
\usepackage{comment}
\usepackage[utf8]{inputenc}
\usepackage[T1]{fontenc}
\usepackage{lmodern}
\usepackage[main=english]{babel}
\usepackage[type={CC},modifier={by-nc-sa},version={3.0}]{doclicense}
\usepackage{amsthm}
\usepackage{amssymb}
\usepackage{amsfonts}
\usepackage{stmaryrd}
\usepackage{bbold}
\usepackage{mathrsfs}
\usepackage{fancyhdr}
\usepackage{fancyref}
\usepackage{float}
\usepackage{csquotes}
\usepackage{quiver}
\usepackage{mathtools}
\usepackage[style=alphabetic, mincitenames=1, maxnames=4, giveninits, backend=bibtex]{biblatex}

\AtEveryBibitem{
   \clearfield{month}
   \clearfield{series}
   \clearfield{venue}
   \clearname{editor}
   \clearlist{publisher}
   \clearlist{location} % alias to field 'address'
   \clearfield{doi}
   \clearfield{url}
   \clearfield{venue}
   \clearfield{issn}
   \clearfield{isbn}
   \clearfield{urldate}
   \clearfield{eventdate}
}

\newtheorem{numbering}{numbering}[section]
\numberwithin{numbering}{section}

\theoremstyle{plain}
\newtheorem{thm}[numbering]{Theorem}
\newtheorem{pps}[numbering]{Proposition}

\newtheorem{csq}[numbering]{Corollary}
\newtheorem{lem}[numbering]{Lemma}
\newtheorem{conj}[numbering]{Conjecture}

\theoremstyle{definition}

\newtheorem{df}[numbering]{Definition}

\theoremstyle{remark}
\newtheorem{rmq}[numbering]{Remark}

\newtheorem{ex}[numbering]{Example}

\numberwithin{equation}{section}

\DeclareMathOperator{\identity}{id}

\DeclareMathOperator{\End}{End}

\DeclareMathOperator{\dyn}{dyn}
\DeclareMathOperator{\Ad}{Ad}
\DeclareMathOperator{\ad}{ad}
\DeclareMathOperator{\Hol}{Hol}
\DeclareMathOperator{\Or}{O}
\DeclareMathOperator{\SO}{SO}

\DeclareMathOperator{\CO}{CO}
\DeclareMathOperator{\GL}{GL}
\DeclareMathOperator{\PGL}{PGL}

\DeclareMathOperator{\U}{U}
\DeclareMathOperator{\SU}{SU}
\DeclareMathOperator{\Sp}{Sp}
\DeclareMathOperator{\Spin}{Spin}

\DeclareMathOperator{\PSU}{PSU}
\DeclareMathOperator{\PSL}{PSL}
\DeclareMathOperator{\PSO}{PSO}

\newcommand{\R}{\mathbb{R}}
\newcommand{\C}{\mathbb{C}}
\newcommand{\HH}{\mathbb{H}}

\newcommand{\blue}[1]{\textcolor{blue}{#1}}

\newcommand{\frk}[1]{\mathfrak{#1}}

\title[On Kanai's conjecture for frame flows]{On Kanai's conjecture for frame flows over negatively curved manifolds}
\author{Louis-Brahim Beaufort}

\address{Louis-Brahim Beaufort, Université Paris-Saclay,  Laboratoire de mathématiques d'Orsay, 91405, Orsay, France}

\email{louis-brahim.beaufort@math.cnrs.fr}



\begin{document}

\begin{abstract}
	Let $M$ be a closed, negatively curved Riemannian manifold of dimension $n \neq 4, 8$ with strictly $1/4$-pinched sectional curvature. We prove, that if the frame flow is ergodic and the sum of its unstable and stable bundles together with its flow direction is $\mathcal{C}^2$, then $M$ is homothetic to a real hyperbolic manifold. This extends to higher dimensions a previous result of Kanai \cite{kanai1993differential} in dimension 3. The proof generalises to isometric extensions of geodesic flows to a principal bundle $P$ with compact structure group and yields the following alternative : either $P$ is flat, or $M$ is hyperbolic.
\end{abstract}

\maketitle

\section{Introduction} \label{sec:intro}

\subsection{Context}

Associated to the geodesic flow $\phi_t$ on the unitary tangent bundle $SM$ of a closed Riemannian manifold $M$ with negative sectional curvature are the so-called stable and unstable bundles $E_s, E_u$. The regularity of these bundles is an old, almost completely settled problem featuring a smooth rigidity property first remarked by Hurder and Katok \cite{hurder1990diff} over surfaces: if the stable and unstable bundles $E_s, E_u$ are $\mathcal{C}^2$, then they are automatically smooth, and the underlying manifold is locally symmetric as proved by Ghys \cite{ghys1987flots}. Kanai \cite{kanai1988geodesic} and Feres \cite{feres1991geodesic} showed in higher dimensions with a pinching assumption that if $E_s, E_u$ are smooth, then the geodesic flow $\phi_t$ is conjugate to the geodesic flow $\hat{\phi}_t$ of a locally symmetric space; at the same time Benoist-Foulon-Labourie \cite{benoist1990flots} obtained the same result for more general contact Anosov flows under a (conjecturally non-optimal) bound on the regularity of $E_s, E_u$ depending on the dimension. It then follows from a celebrated result of Besson-Courtois-Gallot \cite{besson1995entropies} that the existence of such a conjugacy implies that the underlying manifold is locally symmetric. Note that compact locally symmetric spaces of negative curvature have been classified: they are precisely compact quotients by lattices of one of the real, complex, quaternionic hyperbolic spaces or the Cayley plane.

\vspace{3mm}

One might then ask if similar results hold when considering other related geometrical flows; a natural candidate is the frame flow $\Phi_t$, defined on the frame bundle $FM$. By definition, $\Phi_t$ is the parallel transport of frames $(x, e_1, \dots, e_n)$ of tangent vectors along the geodesic directed by $v \coloneq e_1$ for a time $t$, see Section \ref{sec:ext_flow}. There is a natural $\SO(n-1)$ principal fibration $\pi: FM \to SM$ defined by $\pi(x, v, e_2, \cdots, e_n) = (x, v)$, such that $\pi \circ \Phi_t = \phi_t \circ \pi$. Note that $\Phi_t$ acts fiberwise isometrically, hence is not Anosov; however it is still \emph{partially hyperbolic} \cite{brin1974partially,hasselblatt2006partially},\cite[Ch. 12]{lefeuvre2024microlocal} in the sense that there exists a continuous decomposition of the tangent bundle of $FM$ :
$$TFM = \R X^{FM} \oplus E_s^{FM} \oplus E_u^{FM} \oplus \mathbb{V}$$
where $\mathbb{V} = \ker d\pi$ is the \emph{vertical distribution}, $X^{FM}$ is the generator of $\Phi_t$ and $E_s^{FM}$, $E_u^{FM}$ are the stable and unstable bundles of the frame flow, which uniformly contract and expand respectively under the action of $d\Phi_t$. 

The existence of this partially hyperbolic structure was first noticed by Brin and Pesin \cite{brin1974partially} and used successfully in the study of the ergodicity of frame flows \cite{brin1975extensions,brin1980ergodic,brin1982ergodic,brin1984frame,burns2003stable,cekic2024frame,cekic2024unitary}.

The associated smooth rigidity problem was first studied by Kanai in \cite{kanai1993differential}, where he conjectured the following:

\begin{conj}[Kanai] If the distribution 
$$\HH = \R X^{FM} \oplus E_s^{FM} \oplus E_u^{FM}$$
is $\mathcal{C}^1$, then $(M, g)$ is locally symmetric.
\end{conj}
Note the different regularity bound than for the geodesic flow over surfaces, which is $\mathcal{C}^2$. In fact, it should be expected in general that when considering this differentiability problem, the regularity of the bundles on a principal bundle should be lower than on an associated bundle. This is because lifting a connection from a vector bundle to the associated frame bundle uses the existence of parallel local sections, which are in general less regular than the connection itself.

Additionally, for the frame flow, Kanai considers instead the \emph{sum} of the stable and unstable bundles (together with the flow direction) instead of each bundle separately. Note that if $E_s, E_u$ are $\mathcal{C}^k$ on the frame bundle, then their projection on the unitary tangent bundle are $\mathcal{C}^k$ too. Moreover, the projection of $E_s \oplus E_u$ is always smooth for geodesic flows.

\vspace{3mm}

The problem received a positive answer in dimension 3 by the work of Kanai \cite{kanai1993differential} as well as Besson-Courtois-Gallot. Recall that for $\delta > 0$, we say that the curvature of $(M, g)$ is $\delta$-pinched if the sectional curvatures $\kappa_g$ of $M$ satisfy
$$-C \leq \kappa_g \leq -C \delta$$
for some positive constant $C$ (which may, by an homothety, be taken equal to 1) and strictly $\delta$-pinched if it is $\delta'$ pinched for some $\delta' > \delta$. Then, Kanai proved that if the curvature of $M$ is strictly $1/4$-pinched, then the distribution $\HH$ is $\mathcal{C}^2$ if and only if $M$ has constant curvature.

More precisely, Kanai shows that under these conditions, there exists a conformal structure on $E_u$ which is flow-invariant, and proves that this implies that the geodesic flow $\phi_t$ is conjugate to the geodesic flow $\hat{\phi}_t$ on a locally symmetric manifold. As a consequence, the underlying manifold is real hyperbolic by \cite{besson1995entropies}.

The main tool in studying regularity of stable/unstable bundles are dynamical connections. We distinguish here two connections. The first one, introduced by Kanai in \cite{kanai1988geodesic}, is a linear connection on $TSM$ naturally associated to the geodesic flow $\phi_t$ and played a crucial role in the resolution of the smooth rigidity problem for Anosov flows \cite{benoist1990flots}. After successfully solving the aforementioned problems, this connection also found other applications in continuous as well as discrete contexts, see for instance \cite{hamenstadt1993regularity, kalinin2013cocycles, butler2018rigidity, butler2018measurable}. 

The second one is the principal connection of $FM$ with horizontal space $\HH$. Although it has first been introduced in this form by Kanai in \cite{kanai1993differential}, the associated parallel transport had already been investigated by Brin (see \cite{brin1975extensions}) and used since then in the study of ergodicity of the frame flow \cite{brin1975extensions,brin1980ergodic,brin1982ergodic,brin1984frame,burns2003stable,cekic2024frame,cekic2024unitary}.

\subsection{Main result}

In this article, we study the differentiability of dynamical connections in a general framework, extending Kanai's foundational work. Although dynamical holonomies have been investigated since the 1970s, only a few works (such as \cite{brin1980ergodic, cekic2024frame, cekic2024unitary}) have focused on the holonomy group generated by these maps as a central object and analyzed its properties directly. This paper aims to provide a unified reference for the definition and analysis of such holonomy groups and their associated holonomy bundles under a bunching assumption. This setting, which is currently the best understood, encompasses isometric extensions of Anosov flows (see \cite{lef2023isometric}) as well as geodesic flows on strictly $1/4$-pinched manifolds.

By applying our theory to the frame flow, we prove the following:
\begin{thm} \label{thm_conj_Kanai}
Let $M$ be a Riemannian manifold of negative, $\delta$-pinched curvature for some $\delta > 1/4$ and dimension $n \geq 3$. Assume that the associated horizontal distribution $\HH$ is $\mathcal{C}^{1/\sqrt{\delta}}$.
\begin{itemize}
   \item If $n$ is odd, then $M$ is homothetic to a real hyperbolic manifold.
   \item If $n$ is even, $n \neq 4, 8$, and the frame flow of $M$ is ergodic, then $M$ is homothetic to a real hyperbolic manifold.
   \item If $n = 4$ or $8$ and the frame flow of $M$ is ergodic, then either $M$ is homothetic to a real hyperbolic manifold or $\HH$ is integrable.
\end{itemize}

\end{thm}
The $\mathcal{C}^{1/\sqrt{\delta}}$ regularity approaches, as $\delta \to 1$, the conjectural regularity $\mathcal{C}^1$. 

In dimension $4, 8$ we expect that $\HH$ is not integrable. However, we cannot prove this at the moment. Indeed, to show that $\HH$ is not integrable in other dimensions, we restrict the bundle $FM \to SM$ to a single sphere and we use that the frame bundle $FS^{n-1} \to S^{n-1}$ over the sphere is not a flat bundle except when $n = 2, 4, 8$. However, $S^3, S^7$ are parallelizable, which prevents us from concluding in these cases.

Note that ergodicity of the frame flow always holds in odd dimensions different from 7 \cite{brin1980ergodic}, and is known in even dimension and dimension 7 under some pinching. The best pinching currently known is obtained in \cite{cekic2024frame}; for $n = 2 \pmod{4}$, ergodicity holds for $\delta$-pinching with $\delta \sim 0.28$, and for $n = 0 \pmod{4}$, for $\delta \sim 0.56$. Remark that in Theorem \ref{thm_conj_Kanai}, we do not require ergodicity of the frame flow for $n = 7$.

In dimension 3, Theorem \ref{thm_conj_Kanai} only improves the regularity in Kanai's result. Moreover, it is not so easy to adapt our formalism to this case (more precisely, the issue here is that the Lie algebra $\frk{so}(2)$ is Abelian). Thus, instead of writing a full proof, we only explain how one can obtain it from other results. 

We give in the following two generalizations of Theorem \ref{thm_conj_Kanai} to more general contexts, which both include the frame flow on 3-manifolds as a special case and imply the desired result. We say more about this below. 

One should also be able to directly generalize the original proof of Kanai in dimension 3 by using Hamenstädt's results \cite{hamenstadt1995invariant} on \emph{continuous} flow-invariant 2-forms, which strengthen Kanai's which only hold for $\mathcal{C}^1$ regular 2-forms. 

Alternatively, one could use analysis and adapt \cite{bonthonneau2021RadialSourceEstimates} to partially hyperbolic flows in order to show that if the distribution $\mathbb{H}$ is $\mathcal{C}^{1/\sqrt{\delta}}$, then it is automatically smooth; then apply Kanai's original proof to conclude.

\subsection{Generalization}

Using similar methods, we are able to generalize Theorem \ref{thm_conj_Kanai} to a larger class of flows. We are interested in flows $\Phi^P_t$ on a principal bundle $P \to SM$ with compact structure group $G$ such that $\Phi_t^P$ commutes with the right action of $G$, acts isometrically between fibers and projects to $\phi_t$ on the base $SM$ (see Section \ref{sec:ext_flow}).

Our first result hinges on irreducibility assumptions on both the structure group and the flow:

\begin{thm} \label{thm_conj_Kanai_simple}
Assume that $M$ has negative, $\delta$-pinched curvature for some $\delta > 1/4$, $\Phi^P_t$ is \emph{ergodic} and the structure group $G$ is a \emph{simple} Lie group of dimension at least $(n-1)(n-2)/2$. 

If the Kanai horizontal distribution is $\mathcal{C}^{1/\sqrt{\delta}}$, then either the Kanai horizontal distribution is integrable or $M$ is homothetic to a real hyperbolic manifold.
\end{thm}

Note that this result cannot be applied to the frame flow on 3 and 5-dimensional manifolds, as the structure groups $\SO(2), \SO(4)$ are not simple. 

We are also able to generalize this result to extensions of arbitrary Anosov flows $\varphi_t$ by connecting our differentiability problem for extensions of $\varphi_t$ to the problem of differentiability for the stable and unstable bundles of $\varphi_t$ itself. See Theorem \ref{thm_general_simple} for details.

\smallskip

Our second result is based on topological properties of even-dimensional spheres. Recall that a principal bundle is said to be \emph{flat} if it carries a flat connection.
\begin{thm} \label{thm_conj_Kanai_odd}
Assume that $M$ has odd dimension and has negative, $\delta$-pinched curvature for some $\delta < 1/4$. If the  distribution $\HH$ is $\mathcal{C}^{1/\sqrt{\delta}}$, then either $M$ has constant curvature or the bundle $P \to SM$ is flat. 

More specifically, in this latter case, the flow $\Phi_t$ may be described explicitly as a "homothetical curvature flow" (see Section \ref{sec:exp_flow}).
\end{thm}
It should be noted that examples of such flows with $P$ flat but nonintegrable Kanai horizontal distribution do exist, see Section \ref{sec:exp_flow}.

Theorem \ref{thm_conj_Kanai} can be obtained as a consequence of the second and third results when the frame bundle $FM$ can be proved to be non-flat as a bundle over $SM$ as explained above.

\subsection{Strategy of proof} 

Denote $SM$ the unitary tangent bundle of $M$, $\phi_t$ the geodesic flow, $E_s, E_u \subset TSM$ the stable and unstable bundles of the geodesic flow. Under the regularity assumption on $\mathbb{H}$, we will prove that there is a $\phi_t$-invariant conformal structure on $E_u$, which implies, by \cite{kanai1993differential} and \cite{besson1995entropies}, that $M$ has constant curvature. To do this, we study dynamical parallel transport on $E_u$ and more precisely the associated holonomy group $H$ acting on a fiber $E_u \vert_v$. We show that $H$ preserves a conformal structure on $E_u \vert_v$, which then extends to every fiber of $E_u$ by parallel transport.

This is done as follows. If the dynamical connection on $FM$ associated to the frame flow is $\mathcal{C}^1$, then it admits a well-defined curvature. One can identify this tensor to a morphism $F$ of vector bundles between the adjoint bundle $\Ad(P)$ and the bundle $\End(E_u)$. Moreover, this curvature is invariant under a natural flow obtained from the frame flow. Using this property, we show that the curvature is parallel in some sense; more precisely it is invariant under a parallel transport associated to dynamical connections on $\Ad(P)$ and $\End(E_u)$. We leverage this equivariance property to show that the holonomy group $H$ must contain the orthogonal group associated to some metric $h$ on $E_u$. Finally, by using the representation theory of $\SO(n-1)$, we are able to conclude that $H$ must preserve the conformal structure associated to $h$. 

\subsection{Possible generalizations} 

As explained above, a key element in our proof is the existence of holonomies acting on the stable and unstable bundles. Such holonomies have been constructed by Kanai \cite{kanai1988geodesic,kanai1993differential} when the stable bundle is $\mathcal{C}^1$, including strictly $1/4$-pinched manifolds and locally symmetric manifolds. Other methods have been devised under bunching assumptions on the flow \cite{viana2008cocycles, kalinin2013cocycles, sadovskaya2014cohomology}, giving another proof of existence for geodesic flows on strictly $1/4$-pinched manifolds (but not beyond). However, as of now, we remain unaware of any other criterion enabling the existence of such holonomies. 

Notably, the pinching assumptions exclude Kähler and quaternion-Kähler manifolds, which possess their own specialized frame flows (acting on unitary and quaternionic frames, respectively). It has been observed that dynamical holonomies must be adapted to the geometric structure of the underlying manifold; see \cite{benoist1990flots,feres1991geodesic}. Our approach to constructing dynamical holonomies, which currently requires strict $1/4$-pinching of the curvature, thus implicitly leverages the favorable properties present in the real hyperbolic case. This indicates that a completely new construction of dynamical holonomies may be necessary to fully resolve Kanai's conjecture and its analogues for special frame flows.

Nevertheless, our methods could likely be extended to address generalizations of Theorem \ref{thm_conj_Kanai_odd} in even dimensions, under strict 1/4 pinching. Importantly, our arguments do not depend on ergodicity assumptions for the flow, suggesting that a proof of Theorem \ref{thm_conj_Kanai} without requiring ergodicity is within reach.

This approach relies on the topological properties of spheres to compute explicitly which groups can appear as holonomy groups for our various holonomies. More precisely, reductions of the frame bundle over the sphere $S^{n-1}$ have strong constraints on their structure group, which get looser as the 2-adic valuation of $n$ increases, peaking with $S^3, S^7$ which are parallelizable. We estimate that our methods are tractable when $n \neq 0 \pmod{4}$, and with much effort maybe $n \neq 0 \pmod{8}$, but a complete study of the octonionic hyperbolic plane $\HH^2(\mathbb{O})$ seems already out of reach without a computer. We also note that the study of the reductions of frame bundle over the sphere has only been completely solved in dimensions not divisible by 4 (except for a finite number of exceptions) \cite{leonard1971G,ozaki1991g}, and progress may also be needed in this direction.

\subsection{Organization of the paper}

We gather in Section \ref{sec:background} some facts and definitions about the geodesic flow, representation theory and algebraic topology on the spheres. Section \ref{sec:dyn_connections} is dedicated to dynamical connections. We first construct stable holonomies in the general context of fiber bunched extensions of Anosov flows. In the more specific context of isometric extensions, holonomies actually come from a connection. In Section \ref{sec:curvature}, we study the associated curvature when it is defined. In Section \ref{sec:proof_kanai} we apply this theory, ultimately proving Theorem \ref{thm_conj_Kanai}. Finally, we generalize the proof in several ways in Section \ref{sec:generalizations}.

\bigskip

\paragraph{\textbf{Conventions and notations}} \ 

If $k = m + \alpha \in \R$ with $m \in \mathbb{N}$ and $\alpha \in [0, 1)$, we denote $\mathcal{C}^k$ or $\mathcal{C}^{m, \alpha}$ the class of $m$ times differentiable functions with $\alpha$-Hölder continuous derivatives of order $m$.

Given a metric $g$ on an Euclidean space, we denote $\SO(g)$ the corresponding special orthogonal group and $\CO(g) \coloneq \mathbb{R}^* \SO(g)$ the corresponding conformal group. If $g$ is the standard metric on $\R^n$, we simply write $\SO(n)$ and $\CO(n)$ instead of $\SO(g)$ and $\CO(g)$. Similarly, we denote $\U(n), \SU(n)$ the unitary and special unitary groups associated to the standard metric and complex structure on $\R^{2n}$. 

Given a matrix group $H \subset \GL$, we denote $PH$ the image of $H$ in the projective group $\PGL$.

When referring to the ergodicity of an isometric, compact extension of a measure-preserving Anosov flow, we shall always consider ergodicity with respect to the measure obtained as the product of the preserved measure on the base and the Haar measure on the fibers. Equivalently, one can replace ergodicity by transitivity, as these are known to be equivalent.

\bigskip

\paragraph{\textbf{Acknowledgements}} \ This work is part of the author's PhD thesis. The author wishes to thank Thibault Lefeuvre and Andrei Moroianu for suggesting the problem and their constant support during the elaboration of this article.

This research was supported by the European Research Council (ERC) under the European Union's Horizon 2020 research and innovation programme (Grant agreement no. 101162990 -- ADG). 

\section{Background} \label{sec:background}

\subsection{Geodesic flows in negative curvature} \label{sec:geod_flow}

We refer the reader to \cite{paternain1999geodesic} or \cite[Ch. 13]{lefeuvre2024microlocal} for an introduction to geodesic flows and Anosov geodesic flows. Let $(M, g)$ be a closed oriented Riemannian manifold of dimension $n$. Denote $SM$ the unit tangent bundle of $M$, and $\phi_t$ the geodesic flow of $g$ over $SM$. 

We say that $\phi_t$ is an Anosov flow if there exist a Riemannian metric and a continuous decomposition 
$$TSM = \R X^{SM} \oplus E_u^{SM} \oplus E_s^{SM}$$
where $X^{SM}$ is the generator of $\phi_t$, $E_u^{SM}$ and $E_s^{SM}$ are the unstable and stable bundles of the geodesic flow, satisfying the following property: there exist positive constants $C, c > 0$ such that 
$$\forall t > 0, \Vert d\phi_t \vert_{E_s^{SM}} \Vert \leq C e^{- ct}, \quad \Vert d\phi_{-t} \vert_{E_u^{SM}} \Vert \leq C e^{- c t}.$$
As $M$ is closed, this definition does not depend on the choice of a metric.

The celebrated stable/unstable manifolds theorem then asserts that the bundles $E_s^{SM}, E_u^{SM}$ integrate to continuous foliations with smooth leaves. In the following, if $v \in N$, we shall denote $W^s(v)$ the stable leaf of $v$ and $W^u(v)$ the unstable leaf of $v$.

It is well known that when the sectional curvatures of $M$ are negative, $\phi_t$ is Anosov. Moreover, $\phi_t$ preserves a contact 1-form $\lambda$ (the Liouville form) characterized by $\lambda(X^{SM}) = 1, \ker \lambda = E_s^{SM} \oplus E_u^{SM}$. This means $\lambda \wedge d\lambda^{n-1}$ is a volume form, in particular the restriction of $d\lambda$ to $E_s \times E_u$ defines a nondegenerate pairing on each pair of fibers.

The Levi-Civita connection associated to $g$ is denoted $\nabla^{LC}$. There is an associated decomposition in the tangent bundle of $SM$:
$$TSM = \R X^{SM} \oplus H \oplus V$$
where $V$ is the kernel of the footpoint projection $p: SM \to M$ and $H$ is the horizontal distribution, defined through the parallel transport of $\nabla^{LC}$. (Note that under the usual convention, the horizontal space should be $\R X \oplus H$, however we shall instead call $\R X \oplus H$ the \emph{total horizontal space}.) The projection $dp$ restricts to an isomorphism $\R X_v \oplus H_v \to T_{p(v)}M$ between fibers. By definition, $H_v$ projects to $v^\perp \subset T_{p(v)}M$ under through $dp$. 

Consider now the projection $p: SM \to M$. The tangent bundle $TM \to M$ pulls back to $p^* TM \to SM$. This bundle has a tautological section $s(x, v) = v$. The \emph{normal bundle} $\mathcal{N}$ of $SM$ is then defined by
$$(p^* TM)_{(x, v)} = \R s(x, v) \overset{\perp}{\oplus} \mathcal{N}_{(x, v)}.$$
Over each sphere $S_xM \cong S^{n-1}$, $\mathcal{N}$ is naturally identified to the tangent bundle $TS^{n-1}$.

It is a famous theorem of Klingenberg \cite{klingenberg1974anosov} that when the geodesic flow is Anosov, then $g$ has no conjugate points and moreover 
$$E_u^{SM} \cap V = \{0\} = E_s^{SM} \cap V. \label{eq:klingenberg}$$
As a result, the bundles $E_{s, u}^{SM}$ are naturally isomorphic to $\mathcal{N}$ through the projection $dp$. 

\smallskip 

There is also an identification of $V$ with $\mathcal{N}$, given by the so-called connection map $\mathcal{K}$ defined as follows. For $(x, v) \in SM$ and $\xi \in T_{(x, v)} SM$, let $z(t) = (x(t), v(t))$ be a germ of curve at $(x, v)$ on $SM$ such that $x'(0) \neq 0$, $z'(0) = \xi$. Then, we define
$$\mathcal{K}_{(x, v)} \xi = \nabla^{LC}_{x'(t)} v(t) \vert_{t = 0} \in \mathcal{N}_{(x, v)}.$$
One checks that $K$ is linear in $v, \xi$ and $\ker \mathcal{K}_{(x, v)} = \R X_{(x, v)} \oplus H_{(x, v)}$ so that $\mathcal{K}$ induces a bundle isomorphism between $V$ and $\mathcal{N}$.

This lets us define a natural metric on $TSM$ called the Sasaki metric $g_{Sas}$:
$$g_{Sas}(\xi, \eta) = g(dp(\xi), dp(\eta)) + g(\mathcal{K}(\xi), \mathcal{K}(\eta))$$
with the property that $X, H$ and $V$ are orthogonal subspaces.

The horizontal and vertical subspaces give us convenient coordinate systems to work with the geodesic flow. More precisely, consider the action $\xi(t) = d\phi_t(x, v) \xi$ for some $(x, v, \xi) \in TSM$. If $\xi$ is orthogonal to $X_{(x, v)}$, then one checks that $\xi(t)$ remains orthogonal to $X_{\phi_t(x, v)}$ and writes 
$$\xi(t) = (J(t), \nabla_{\gamma'(t)}^{LC} J(t)) \in (\mathcal{N} \oplus \mathcal{N})_{(x, v)} \cong (H \oplus V)_{(x, v)}.$$
We say that $J$ is a \emph{Jacobi field}. It is a vector field along the geodesic $\gamma$ generated by $(x, v)$. Jacobi fields satisfy the Jacobi equation 
$$\left (\nabla_{\gamma'(t)}^{LC} \right )^2 J(t) + R_{\gamma(t)}(J(t), \gamma'(t)) \gamma'(t) = 0 \label{eq:jacobi}$$
where $R$ is the Riemann curvature tensor associated to the metric $g$.

\subsection{Extensions of flows} \label{sec:ext_flow}

Let $\varphi_t$ be a flow on a compact manifold $N$. An \emph{extension} of $\varphi_t$ (or cocycle over $\varphi_t$) is a flow $\Phi_t$ on a smooth bundle $\pi: F \to N$ such that $\pi \circ \Phi_t = \varphi_t \circ \pi$ and $\Phi_t: P_v \to P_{\phi_t v}$ is an isomorphism of fibers (in the relevant category). For instance, we may talk about linear, isometric or principal extensions if $P$ is a vector bundle, Riemannian bundle or principal bundle. More explicitly, if $F = P$ is a principal bundle, we say that $\Phi$ is a \emph{principal extension} if moreover $\Phi_t$ commutes with the right action $R_g$ of the structure group $G$ of $P$:
$$R_g \circ \Phi_t = \Phi_t \circ R_g \quad \forall g \in G.$$

For example, $d\varphi_t$ is a linear extension of $\varphi_t$ over $TN$. 

Note that most of the time, we can associate to any extension a principal extension. For example, assume that $\Phi_t$ is a linear extension of $\varphi_t$ over some vector bundle $V$. Then, the associated principal bundle is the frame bundle $\GL(V)$ of $V$ and $\Phi_t$ acts on vectors of the frames componentwise. As a result, we will mostly be interested in principal extensions.

Conversely, a principal extension $\Phi_t$ will induce extensions of $\varphi_t$ over any associated bundle. Recall that the associated bundle to a principal $G$-bundle $P$ and a representation $\rho: G \to \GL(V)$ is the vector bundle $E = P \times_\rho V = P \times V / \sim$ where $(p \cdot g, v) \sim (p, \rho(g)v) \forall g \in G, p\in P, v \in V$. Then a flow $\Phi_t^P$ on $P$ will induce the flow 
$$\Phi_t^\rho [u, \xi] \coloneq [\Phi_t^P u, \xi]$$
on $E$. The flow is well-defined precisely because $\Phi_t^P$ commutes with the right action of $G$.

We now give some examples. 

\begin{ex}[Frame flows] \label{ex:frame_def}
Let $E$ be a vector bundle over $N$ with a connection $\nabla^E$. Then we can associate to it a \emph{frame flow} in the following way: for $v \in N$, consider the parallel transport map $\Phi^E_t(u)= \tau_{v \to \varphi_t(v)} u$ of some $u \in E_v$ along the flowline $t \mapsto \varphi_t(v)$. This defines a flow $\Phi_t^E$ on $E$ which extends to $\GL(E)$. Moreover, if $\nabla^E$ preserves a Euclidean metric $g$, then $\Phi^E$ is isometric, so also extends to the orthonormal frame bundle $F(E, g)$. If $E$ is a complex vector bundle and $\nabla^E$ preserves a Hermitian metric $h$, then $\Phi^E$ is unitary and extends to the unitary frame bundle $F(E, h)$.

More specifically, an example of such a vector bundle is given by the normal bundle $\mathcal{N}$ of the unitary tangent bundle $SM$, equipped with the pullback of the Levi-Civita connection. In this case, the bundle $F\mathcal{N}$ may be identified with the frame bundle $FM$ of $M$ itself in the following way : any $(x, v) \in SM$ and $(e_2, \dots, e_n) \in \SO(\mathcal{N})_{(x, v)}$ give an orthonormal frame $(v, e_2, \dots, e_n)$ of $T_x M$. We will call $\Phi_t^{F\mathcal{N}}$ the frame flow on $M$.
\end{ex}

Recall that given a principal bundle $P \to N$ with structure group $G$ and $H \subset G$ a subgroup, a \emph{reduction of structure group} from $G$ to $H$ is a principal subbundle $Q \subset P$ with structure group $H$ such that the fibers of $Q$ are invariant under the right action of $H$. If $\Phi_t^P$ is a principal extension of $\varphi_t$ and $Q$ is $\Phi_t^P$-invariant, then $\Phi_t^P \vert_Q$ defines a principal extension of $\varphi_t$ on $Q$ with structure group $H$.

Conversely, if $\Phi_t^Q$ is a principal extension of $\varphi_t$ on $Q$, then we can \emph{uniquely} extend it to $P$ by defining $\Phi_t^P(u) = \Phi_t^Q(u)$ for $u \in Q$ and then extending it by translation on the right by $G$.

\subsection{Reductions of the frame bundle of the spheres} \label{sec:g_struct}

Principal bundles on CW-complexes can be classified up to homotopy using tools from algebraic topology. In general, one has to investigate the classifying space associated to the structure group of the principal bundle. On spheres (and more generally on suspensions) the problem can be stated in a simpler way because a principal bundle on the sphere $S^m$ is described by its so-called \emph{clutching map} $\tau: S^{m-1} \to G$ (or more precisely by the homotopy class of this mapping) which specifies how two trivial bundles on each hemisphere can be glued, up to homotopy.

The more geometrical question of finding whether a principal $G$-bundle admits a reduction to a subgroup $H$ can then be slightly generalized in this topological context.
\begin{df}
Let $P \to S^m$ be a continuous principal b$G$-undle over $S^m$, where $G$ is a Lie group. We say that $P$ admits a reduction of its structure group to $(H, \rho)$, where $H$ is a Lie group and $\rho : H \to G$ is a homomorphism, if the clutching map $\tau: S^{m-1} \to G$ can be factored (up to homotopy) through the map $\rho$, that is, there exists a clutching map $\tau_0$ such that the following diagram commutes up to homotopy:
% https://q.uiver.app/#q=WzAsMyxbMCwxLCJTXntuLTF9Il0sWzEsMCwiSCJdLFsxLDEsIkciXSxbMCwyLCJcXHRhdSIsMl0sWzAsMSwiXFx0YXVfMCJdLFsxLDIsIlxccmhvIl1d
\[\begin{tikzcd}
	& H \\
	{S^{m-1}} & G
	\arrow["\rho", from=1-2, to=2-2]
	\arrow["{\tau_0}", from=2-1, to=1-2]
	\arrow["\tau"', from=2-1, to=2-2]
\end{tikzcd}\]
\end{df}

In particular, if $H$ is a closed subgroup of $G$, any reduction of $P$ as a principal $G$-bundle to a principal $H$-subbundle defines a reduction of structure group from $G$ to $H$.

Note that we do not require in the definition that $H$ is a subgroup of $G$ and $\rho: H \hookrightarrow G$ is the inclusion; part of the problem of finding which $H$ can be reductions of the structure group include not only finding $H$ but also the morphism $\rho: H \to G$.

We will need the following fact:
\begin{thm}[{\cite[Theorem 1]{leonard1971G}}] \label{thm_leo_odd}
Let $n = 2k + 1 \geq 3$ be odd. Then, the $\SO(n-1)$ structure over $S^{n-1}$ induced by the principal frame bundle $F S^{n-1} \to S^{n-1}$ does not admit any reduction to a proper closed subgroup $H$ of $\SO(n-1)$, except if $n = 7$ and $H$ is $\U(3)$ or $\SU(3)$.
\end{thm}

\section{Dynamical holonomies} \label{sec:dyn_connections}

\subsection{Dynamical parallel transport} \label{sec:stable_hol}

Let $(P, \Phi_t)$ be a linear or principal extension of some Anosov flow $\varphi_t: N \to N$ with $N$ closed (see Section \ref{sec:ext_flow}). Under specific conditions, one can construct a notion of dynamical parallel transport along some particular curves associated to the flow $\Phi_t$.

\begin{df}
A \emph{stable holonomy} for $\Phi_t$ is the data of \emph{stable holonomy maps} $\Pi^s_{v \to w}$ defined for every $v \in N, w \in W^s(v)$, satisfying the following properties :
\begin{enumerate}
\item[(1)] $\Pi^s_{v \to w} : P_v \to P_w$ is a continuous and either linear or $G$-equivariant map (depending on the type of the extension $P$);
\item[(2)] $\Pi^s_{v \to v} = \identity$ and $\Pi^s_{v \to w} = \Pi^s_{z \to w} \circ \Pi^s_{v \to z}$ for every $z, w \in W^s(v)$;
\item[(3)] $\Phi_t \circ \Pi^s_{v \to w} = \Pi^s_{\varphi_t(v) \to \varphi_t(w)} \circ \Phi_t$. 
\end{enumerate}
Unstable holonomies $\Pi^u$ are defined similarly.
\end{df}
We also denote $\Pi^c_{v \to \varphi_t(v)} = \Phi_t \vert_{P_v}$ the \emph{central holonomies}.

\begin{rmq}
It is unclear for now whether these are genuine holonomies and if they exist, whether they are induced by some continuous connection. However, when they are defined, if $w \in W^s(v)$, then the holonomy map $P_v \to P_w$ does not depend on a choice of path. In other words, what we will want to call a parallel transport will be flat when restricted to paths in $W^s(v)$. Moreover, there is always a well-defined covariant derivative $\nabla^{\dyn}_X$ in the direction of the generator $X$ of $\Phi_t$, which is precisely $\mathscr{L}_X$.
\end{rmq}

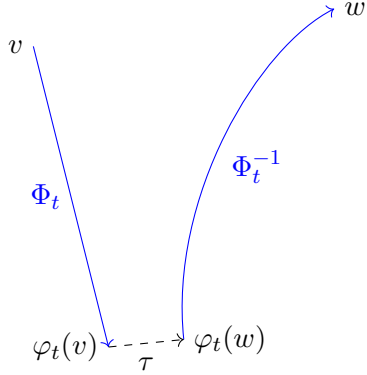
\begin{figure}[!h]
\begin{tikzpicture}
\draw [blue, ->] (0, 0) -- (1, -4);
\draw [blue, ->] (2, -3.9) ..controls (1.8, -2) and (3, 0).. (4, 0.5);
\draw [dashed, ->] (1, -4) -- (2, -3.9);

\draw (0, 0) node[left]{$v$};

\draw (0.5, -2) node[left]{$\blue{\Phi_t}$};
\draw (4, 0.5) node[right]{$w$};

\draw (2.5, -1.6) node[right]{$\blue{\Phi_t^{-1}}$};
\draw (1, -4) node[left]{$\varphi_t(v)$};

\draw (2, -3.9) node[right]{$\varphi_t(w)$};

\draw (1.5, -4) node[below]{$\tau$};
\end{tikzpicture}
\caption{Defining stable holonomies: first go along the flow for a time $t$, then switch to the flowline of $w$, then go back to $w$ by following the flow for the same time $t$.}
\label{fig:stable_hol}
\end{figure}

To define stable holonomies, we will crucially rely on the fact that the flow contracts in the stable direction. More precisely, remark that if $v \in N$ and $w \in W^s(v)$, then constructing the stable holonomy $\Pi^s_{v \to w}$ is equivalent to constructing $\Pi^s_{\varphi_t(v)\to \varphi_t(w)}$ for any $t$. However, remark that if $t$ is large, then the base points $\varphi_t(v)$ and $\varphi_t(w)$ are very close to each other; thus one should be able to identify the fibers $P_{\varphi_t(v)}$ and $P_{\varphi_t(w)}$ by "letting $t \to +\infty$". 

More precisely, stable holonomy would be defined as a limit of the form
$$\Pi^s_{v \to w} = \lim_{t \to + \infty} \Phi_{-t} \circ \tau_{\varphi_t(v) \to \varphi_t(w)} \circ \Phi_t$$
where $\tau$ will end up being any reasonable (this will be made precise later) way to identify $P_{\varphi_t(v)}$ with $P_{\varphi_t(w)}$. This process is pictured in Figure \ref{fig:stable_hol}.

To do this, we will need some assumption on the flows themselves to ensure that the limit exists; for instance it is unclear whether the limit exists when $\Phi_t$ contracts and/or expands along the flowlines. We follow the presentations of \cite{kalinin2013cocycles} and \cite{besson2025rigidity}.

\begin{df}
We say that a flow $\varphi_t$ is \emph{partially hyperbolic} with respect to a metric $g$ on $N$ if there exists a nontrivial \emph{dominated decomposition}
$$V = E_s \oplus E_c \oplus E_u$$
and $\nu, \hat{\nu}, \gamma, \hat{\gamma} \in \mathcal{C}^0(N)$ such that $\nu < 1, \hat{\nu} < 1$ and for any $t> 0$, $v \in N$ and $e_s \in E_s(v), e_c \in E_c(v), e_u \in E_u(v)$ of norm 1, we have
$$\Vert d\varphi_t(e_s) \Vert < \nu(x)^t < \gamma(x)^t < \Vert d\varphi_t(e_c) \Vert < \hat{\gamma}(x)^{-t} < \hat{\nu}(x)^{-t} < \Vert d\varphi_t(e_u) \Vert.$$
If moreover we can take $\nu, \hat{\nu} < \gamma \hat{\gamma}$, we say that $\Phi_t$ is \emph{center-bunched}.
\end{df}
As $N$ is compact, this definition does not depend on the choice of metric.

Anosov flows are immediately partially hyperbolic and center-bunched (for any $\gamma, \hat{\gamma} < 1$ sufficiently close to 1), as the center direction is then generated by the flow generator, which is fixed. More generally, every isometric extension $\Phi_t$ of an Anosov flow is partially hyperbolic and center bunched (for the associated metric), with $E_c = \R X \oplus \ker d\pi$.

In the following, we consider a flow $\varphi_t$ on a closed manifold $N$ which is partially hyperbolic and center-bunched for some metric $g$ and a linear extension $(E, \Phi)$ of $\varphi_t$. 

\begin{df}
Let $(E, \Phi)$ be a linear extension of $\varphi_t$ and $g^E$ a metric on $E$. We say that $\Phi$ is \emph{fiber-bunched} if there exists $\beta > 0, C > 0$ such that
$$\forall v \in N, \forall t > 0, \Vert \Phi_t(v) \Vert \cdot \Vert \Phi_{t}(v)^{-1} \Vert \leq C \min \left (\nu(\pi(v))^{-t\beta}, \hat{\nu}(\pi(v))^{-t\beta} \right ).$$
Here, $\Vert \cdot \Vert$ denotes the operator norm induced by the metric $g^E$ on vector bundle maps $E \to E$.
\end{df}
In other words, the discrepancies between expansion and contraction rates of $\Phi_t$ should be dominated by both the expansion and contraction rates of $\varphi_t$.

Recall that for $\delta > 0$, a Riemannian manifold $(M, g)$ has $\delta$-pinched curvature if up to a homothety, the sectional curvatures $\kappa$ satisfy $- 1 \leq \kappa \leq - \delta$.
\begin{lem} \label{lem_diff_geod}
If $M$ has $\delta$-pinched curvature and $g^{E_s}$ is the Sasaki metric restricted to $E_s$ (see Section \ref{sec:geod_flow}), then the restrictions $d\phi_t : E_s \to E_s$ and $d\phi_t : E_u \to E_u$ are fiber-bunched for $\beta = \delta^{-1/2} - 1$. 
\end{lem}
\begin{proof} Standard estimates for Jacobi fields give us, for $v \in SM, t > 0, Y \in E_s(v)$:
$$\Vert Y \Vert e^{-t} \leq \Vert d\phi_t(v) Y \Vert \leq \Vert Y \Vert e^{-\sqrt{\delta} t}.$$
As a result, we can take for bunching $\nu = e^{-\sqrt{\delta}}$. Thus, fiber bunching holds for $\beta$ such that $1 -\sqrt{\delta} \leq \beta\sqrt{\delta}$, or equivalently $\beta \geq \delta^{-1/2} - 1$. This proves the result for $d\phi_t: E_s \to E_s$. The unstable case is similar.
\end{proof}

The following is an adaptation of \cite[Proposition 4.2]{kalinin2013cocycles} for extensions of Anosov flows which does not rely on specific properties of $\varphi_t$ as in \cite{besson2025rigidity}.
\begin{thm} \label{thm_stable_hol}
Let $(E, \Phi_t)$ be a $\mathcal{C}^1$-regular linear extension of some contact Anosov flow $\varphi_t$ on a closed manifold $N$. Let $D$ be an arbitrary connection. For $v, w \in N$ close enough denote $\tau_{v \to w}$ the associated parallel transport along the unique geodesic from $v$ to $w$. 

If $\Phi_t$ is fiber bunched with constant $\beta < 1$, then the following limit exists, does not depend on the choice of $D$ and defines stable holonomies:
$$\Pi^s_{v \to w} u = \lim_{t \to + \infty} \Phi_{-t} \circ \tau^D_{\varphi_t(v) \to \varphi_t(w)} \circ \Phi_t u$$
where $v \in N$, $w \in W^s(v)$ is close to $v$ and $u \in E_v$. The stable holonomies are continuous with respect to the end points. Moreover, $\Pi^s$ satisfies the following uniform boundedness property: for $v, w$ close enough,
\begin{align}
\Vert \Pi^s_{v \to w} - \tau_{v \to w} \Vert \leq C d_s(v, w). \label{eq:unif_bound_hol_tp}
\end{align}
Analogous results hold for unstable holonomies.
\end{thm}
 
Before proving this result, we state an important corollary:
\begin{csq} \label{cor_diff_geod}
Assume that $M$ has negative strictly $1/4$-pinched curvature. Then, the differential of the geodesic flow $d\phi_t$ admits stable and unstable holonomies.
\end{csq}
\begin{proof}
This is a direct consequence of Lemma \ref{lem_diff_geod} and of the fact that when $M$ has negative strictly $1/4$-pinched curvature, then the stable and unstable bundles are $\mathcal{C}^1$, see \cite[Theorem 9.4.15]{fisher2019hyperbolic}.
\end{proof}

The proof of Theorem \ref{thm_stable_hol} relies on the Ambrose-Singer formula \cite[Lemma 20.1.7]{lefeuvre2024microlocal} as well as the following lemma:
\begin{lem} \label{lem_connection}
There exists a differentiable connection $\nabla$ on $E$ such that $\nabla_X$ coincides with the Lie derivative $\mathscr{L}_X$.
\end{lem}
\begin{proof}
Let $\nabla'$ be any differentiable connection on $E$. Define $\nabla$ by
\[
\nabla = \nabla' + \lambda \otimes (\mathscr{L}_X - \nabla'_X),
\]
where $\lambda$ is the contact form. Note that $\nabla$ and $\nabla'$ differ by an $\End(E)$-valued 1-form, so $\nabla$ is indeed a connection. Furthermore, $\nabla$ is differentiable since $\nabla'$, $\lambda$, and $X$ are differentiable. For any section $Y$ of $E$, we have
\[
\nabla_X Y = \nabla'_X Y + \lambda(X) (\mathscr{L}_X Y - \nabla'_X Y) = \mathscr{L}_X Y,
\]
since $\lambda(X) = 1$. Thus, $\nabla_X$ coincides with the Lie derivative $\mathscr{L}_X$ as required.
\end{proof}

We can now prove Theorem \ref{thm_stable_hol}.
\begin{proof}[Proof of Theorem \ref{thm_stable_hol}]
The first step is to construct an adapted neighborhood for $v$. By \cite[Proposition 1.1]{besson2025rigidity} (which holds for general Anosov flows with the same proof), there is a uniform lower bound $\rho > 0$ on the radius of injectivity of $W^s(v)$ for any $v \in N$. Let $\varepsilon > 0$. Choose $V = V_\varepsilon(v)$ a normal neighborhood of $v$ in $W^s(v)$ such that $\nu(w) \leq \nu(v)^{1 - \varepsilon},w \in V$ and the diameter of $V$ is less than $\rho$. For $t > 0$, denote $v_t = \varphi_t(v), w_t = \varphi_t(w)$. We have the following:

\begin{lem} \label{lem_forward_backward}
Let $\varepsilon > 0$. Then for any $T > 0$, one can find $c, C > 0$ such that for $v \in N, w \in V_\varepsilon(v)$, 
$$ \Vert \Phi_{t}(v) \Vert \cdot \Vert \Phi_{t}(w)^{-1} \Vert \leq C(1 + c d_s(v, w)) \nu(w)^{-t\beta}.$$
\end{lem}

\begin{proof}
Using that $\Phi_t$ as well as $\nabla$ parallel transport maps are differentiable, so globally Lipschitz by compactness, we have
\begin{align*}
\Vert \Phi_t(v) \Vert &\leq \Vert \Phi_t(v) - \tau_{\varphi_tw \to \varphi_tv} \Phi_t(w) \tau_{v\to w} \Vert + \Vert \tau_{\varphi_tw \to \varphi_tv} \Phi_t(w) \tau_{v\to w} \Vert \\
&\leq C \Vert \Phi_t(w) \Vert d_s(v, w) \\
&+ (1 + C' d_s(v, w)) \Vert \Phi_t(w) \Vert (1 + C'' d_s(v_t, w_t)) \\
&\leq (1 + c d_s(v, w)) \Vert \Phi_t(w) \Vert.
\end{align*}
Then, 
\begin{align*}
\Vert \Phi_{t}(v) \Vert \cdot \Vert \Phi_{t}(w)^{-1} \Vert &= \frac{\Vert \Phi_{t}(v) \Vert}{\Vert \Phi_{t}(w) \Vert} \Vert \Phi_{t}(w) \Vert \Vert \Phi_{t}(w)^{-1} \Vert \\
&\leq C(1 + c d_s(v, w)) \nu(w)^{-t\beta}
\end{align*}
where we used the fiber bunching.
\end{proof}

Now, we want to ensure that we can stay in such neighborhoods while working with the flow. Let $\gamma: [0, 1] \to V$ be the constant speed geodesic from $v$ to $w$ lying in $V_\varepsilon(v)$. Note that the curve $\varphi_t \circ \gamma$ links $v_t$ and $w_t$ while staying in the unstable leaf $W^s(v)$, so it is longer than the geodesic from $v_t$ to $w_t$. As a result,
\begin{align}
d_s(v_t, w_t) &\leq \int_0^1 \Vert d\varphi_t(\gamma(s)) \gamma'(s) \Vert ds \notag\\
&\leq \int_0^1 \nu(\gamma(s))^t dt \ d_s(v, w) \notag, \\
d_s(v_t, w_t) &\leq \nu(v)^{t(1-\varepsilon)} d_s(v, w). \label{eq:shrink_distance}
\end{align}
For $t$ large enough, as $w \in W^s(v)$, the distance $d(v_t, w_t)$ will be smaller than $\rho$. By replacing $v, w$ by $v_{t_0}, w_{t_0}$ for some $t_0$ large enough, we can assume it is the case for $t > 0$. Moreover, as $N$ is compact, $\nu$ is uniformly continuous so that by taking $t_0$ large enough we can ensure that $\varphi_t(V_\varepsilon(v)) \subset V_\varepsilon(v_t)$ for all $t > 0$.

\vspace{5mm}

We now begin the constructions of stable holonomies. Let $\nabla$ be the connection provided by Lemma \ref{lem_connection}. We will show that the following sequence converges as $t \to + \infty$ :
$$P_t = \tau^\nabla_{w \to v} \circ\Phi_{-t} \circ \tau^\nabla_{v_t \to w_t} \circ \Phi_{t}.$$
Note that this sequence is well-defined when $w \in V_\varepsilon(v)$.

Denote $\gamma: [0, 1] \to V_\varepsilon(v)$ the unique constant speed geodesic in $W^s(v)$ between $v$ and $w$ and $w_t^s = \varphi_t(\gamma(s))$. Note that for $T > 0$, $P_T$ can be interpreted as the $\nabla$-parallel transport along a closed loop $\alpha^T$. Further, there is a homotopy $\Gamma^T = \Gamma\vert_{[0, T] \times [0, 1]}$ where $\Gamma(s, t) = \varphi_t \circ \gamma(s) = w_t^s$ such that $\alpha^T$ is the $\nabla$-parallel transport along the boundary of the image of $\Gamma^T$. More precisely, if $C_\to(s, t)$ is the parallel transport along $r \mapsto \Gamma^T(r, 0), 0 \leq r \leq t$ then along $\Gamma^T(t, r), 0 \leq r \leq s$ and $C_\uparrow(s, t)$ is the parallel transport along $\Gamma^T(0, r), 0 \leq r \leq s$ then along $r \mapsto \Gamma^T(r, s), 0 \leq r \leq t$, then $\Pi_T = C_\uparrow(T, 1)^{-1} C_\to(T, 1)$.

The Ambrose-Singer formula \cite[Lemma 20.1.7]{lefeuvre2024microlocal} then writes 
$$P_T = \identity + \int_0^T \int_0^1 C_\uparrow(s, t)^{-1} \mathbf{F}_\nabla(\partial_t \Gamma(s, t), \partial_s \Gamma(s, t)) C_\to(s, t) ds dt.$$
Thus, we have to prove that the integral on the right converges when $T \to + \infty$.

We now bound the integrand for $t \geq 0, s \in [0, 1]$. We take the metric on $TSM$ to be the Sasaki metric. The curvature $\mathbf{F}$ of $\nabla$ is a continuous tensor on $N$, so its operator norm is bounded by compactness: if $Y_v, Z_v \in E_v$, then there is $C > 0$ independent of $v$ such that 
$$\Vert F_v(Y_v, Z_v) \Vert \leq C \Vert Y_v \Vert \Vert Z_v \Vert.$$
Similarly, the parallel transport maps are differentiable, so they are uniformly bounded as linear maps between fibers. As a result,
\begin{align*}
&\Vert C_\uparrow(s, t)^{-1} \mathbf{F}_\nabla(\partial_t \Gamma(s, t), \partial_s \Gamma(s, t)) C_\to(s, t) \Vert \\
&\quad \lesssim \Vert \Phi_t(w_0^s)^{-1} \Vert \cdot \Vert \partial_t \varphi_t \circ \gamma(s) \Vert \cdot \Vert \partial_s \varphi_t \circ \gamma(s) \Vert \cdot \Vert \Phi_t(v) \Vert.
\end{align*}
Remark that $\partial_t \varphi_t = X$ is unitary for the Sasaki metric. Then,
\begin{align*}
\Vert \partial_s \varphi_t \circ \gamma(s) \Vert &\leq \Vert d\varphi_t(w_0^s) \Vert \cdot \Vert \gamma'(s) \Vert \\
&\leq \nu(w_0^s)^t d_s(v, w)
\end{align*}
as $\gamma$ has constant speed $d_s(v, w)$. Finally, Lemma \ref{lem_forward_backward} lets us bound the rest, so that 
\begin{align*}
&\Vert C_\uparrow(s, t)^{-1} \mathbf{F}_\nabla(\partial_t \Gamma(s, t), \partial_s \Gamma(s, t)) C_\to(s, t) \Vert \\
&\quad \lesssim (1 + c d_s(v, w_0^s)) \nu(w_0^s)^{-t\beta} \nu(w_0^s)^t d_s(v, w) \\
&\quad \lesssim \nu(w_0^s)^{t(1 -\beta)} d_s(v, w)
\end{align*}
using that $d_s(v, w_0^s) \leq d_s(v, w)$. As $\beta < 1$, the integral of $\nu(w_0^s)^{t(1 -\beta)}$ converges and our integrand is integrable. Furthermore, the integrand is dominated by $(\sup_{SM} \nu)^{t(1 -\beta)} d_s(v, w)$ which is independent of the base points, hence the dominated convergence theorem tells us that $P_T$ converges to a continuous parallel transport $P_\infty$. Composing on the left $P_\infty$ by $\tau^\nabla_{v \to w}$ gives us the existence and regularity of $\Pi^s_{v \to w}$. Similarly, composing the Ambrose-Singer identity on the left by $\tau^\nabla_{v \to w}$ and applying the previous bound gives us Equation \eqref{eq:unif_bound_hol_tp}.

It remains to prove that the previous results hold for other connections and does not depend on the choice of connection. Let $D$ be an arbitrary connection, then the fiber bunching and Lemma \ref{lem_forward_backward} imply
\begin{align*}
&\Vert \Phi_{-t} \circ \tau^D \circ \Phi_t - \Phi_{-t} \circ \tau^D \circ \Phi_t \Vert \\
&\quad \leq \Vert \Phi_{-t} \Vert \cdot \Vert \Phi_t \Vert \cdot \Vert \tau^\nabla - \tau^D \Vert \\
&\quad \lesssim \nu(w)^{-t\beta} d_s(v_t, w_t).
\end{align*}
Then, by using the relation \eqref{eq:shrink_distance}, we have 
$$\nu(w)^{-t\beta} d_s(v_t, w_t) \leq \nu(v)^{-t\beta(1 - \varepsilon)} \nu(v)^{t(1 - \varepsilon)} d_s(v, w)$$
so that as $(1 - \beta)(1 - \varepsilon) > 0$ the above converges to zero when $t \to +\infty$.
\end{proof}

\begin{ex}[The geodesic flow on real hyperbolic manifolds] \label{ex:hyp_geod_hol}
The fact that real hyperbolic manifolds $M$ have constant curvature allows us to describe almost completely the holonomy group associated to the differential of the geodesic flow $d\phi_t$. Indeed, the Jacobi equation \eqref{eq:jacobi} writes $\nabla_t^{2} J = J$, so orthogonal Jacobi fields along any path $\gamma$ may be obtained as 
$$J(t) = e^t \tau^{LC}_{\gamma} J(0)^u + e^{-t} \tau^{LC}_\gamma J(0)^s$$
where $\tau^{LC}$ denotes parallel transport with respect to the Levi-Civita connection and $J(0) = J(0)^s + J(0)^u$ is the decomposition of $J(0)$ in the stable and unstable subspaces. 

Note that unstable holonomies are flat along the horospheres. As a result, with the notations of Theorem \ref{thm_stable_hol}, for any $v \in SM, w \in W^s(v), u \in E_s\vert_v$, the following does not depend on $t$:
$$\forall t \geq 0, \ \Pi_{v \to w} u = \Phi_{-t} \circ \tau^{LC}_{\varphi_t(v) \to \varphi_t(w)} \circ \Phi_t u = \tau^{LC}_{v \to w} u.$$
Incidentally, we see that the restriction of $\nabla^{LC}$ to the horospheres is flat (see also \cite{besson2025rigidity}). 

However, let us not forget central holonomies. For $u \in E_s \vert_v$, we have $d\phi_t(v)u = e^{-t} \tau^{LC}_{v \to \phi_t(v)}u$. As a result, dynamical holonomies of elements of $E_s$ (and, symmetrically, of $E_u$) are scalar multiples of holonomies along paths for the Levi-Civita connection.
\end{ex}

We now generalize the construction of dynamical parallel transport to principal bundles.
\begin{csq}
Let $G \subset \GL_k(\R)$ be a closed Lie subgroup, $P$ a principal $G$-bundle over $N$ and $\Phi_t : P \to P$ an extension of $\varphi_t$. Assume that the induced flow on $E = P \times_{G} \R^k$ is fiber bunched with constant $\beta \leq 1$. Then, there exists a dynamical holonomy on $P$ which moreover commutes with the right action of $G$.
\end{csq}
\begin{proof}
Consider the frame bundle $F_{\GL}(E)$. The inclusion map $\iota: G \subset \GL_k(\R)$ induces a closed embedding of principal bundles $P \to F_{\GL}(E)$. We may fix a principal connection $D$ on $P$, which extends to a principal connection on $F_{\GL}(E)$ with the same horizontal space. More precisely, the horizontal space $\HH_u$ at $u \in F_{\GL}(E)$ is $(R_a)_* \iota_* \HH_z$ where $u = \iota(z) \cdot a$ with $z \in P$ and $a \in \GL_k(\R)$.

Then, the flow $\Phi_t^P$ on $P$ induces a flow $\Phi_t^E$ on $E$ which satisfies the assumptions of Theorem \ref{thm_stable_hol} by assumption. As a result, there are well-defined (say) stable holonomy maps on $E$ associated to the connection induced by $D$ on $E$ given by formulas
$$\Pi^E_{v, w} Y = \lim_{t \to + \infty} \Phi_{t}^E(w)^{-1} \circ \tau^D_{\varphi_t(v) \to \varphi_t(w)} \circ \Phi_t^E(v) Y.$$
These maps then act on $F_{\GL}(E)$ by composition on the left, thus commuting with the right action of $\GL_k(\R)$. Moreover, the flow $\Phi_t^E$ extends to $F_{\GL}(E)$, again acting on the left, and stable holonomies on $F_{\GL}(E)$ are thus defined by
$$\Pi^{F_{\GL}(E)}_{v, w} u = \lim_{t \to + \infty} \Phi_{t}(w)^{-1} \circ \tau^D_{\varphi_t(v) \to \varphi_t(w)} \circ \Phi_t(v) u.$$
Then, remark that the inclusion $P \to F_{\GL}(E)$ maps $\Phi_t^P$ to $\Phi_t^{F_{\GL}(E)}\vert_P$, and that for $u \in P \subset F_{\GL}(E)$, $\Phi_{t}(w)^{-1} \circ \tau^D_{\varphi_t(v), \varphi_t(w)} \circ \Phi_t(v) u \in P$, as $D$ preserves $P$ by construction. As a result, since $P$ is closed in $F_{\GL}(E)$ and letting $t \to +\infty$, the stable holonomies $\Pi^{F_{\GL}(E)}$ restrict to stable holonomies on $P$. Moreover, they are invariant under the right action of $G$ as they were invariant under the right action of $\GL_k(\R)$.
\end{proof}

\begin{rmq}
If $p : \hat{N} \to N$ is a Riemannian cover, then the flow and the dynamical holonomies on $P$ lift to the cover $p^* P \to \hat{N}$.
\end{rmq}

\subsection{Dynamical holonomy groups} \label{sec:dyn_hol}

We now consider holonomies along loops. Let $v, w \in N$. Denote $\mathcal{P}_{v, w}$ the set of continuous curves from $v$ to $w$ which are finite concatenations of paths $\gamma^s \subset W^s(\gamma^s(0)), \gamma^c \subset W^c(\gamma^s(1)), \gamma^u \subset W^u(\gamma^c(0))$. Also denote $\mathcal{P}_v = \mathcal{P}_{v, v}$ and $\mathcal{P}_{v, 0}$ the subset of contractible loops. The local product structure theorem tells us that such paths always exist; we say more about this in Lemma \ref{lem_exist_csu}.

Then, stable/unstable/center holonomies define a representation of $\mathcal{P}_v$ as a monoid. Indeed, parallel transport along a loop $\gamma = \gamma_1^s \cdot \gamma_1^c \cdot \gamma_1^u \cdots \gamma_k^s \cdot \gamma_k^c \cdot \gamma_k^u \in \mathcal{P}_v$ defines a linear map $\Pi_P(\gamma) = \Pi^s_{\gamma_1^s} \circ \cdots \circ \Pi^u_{\gamma_k^u} : P_v \to P_v$. 

We now choose an identification of $P_v$ with $G$. The fact that holonomies commute with the right action of $G$ implies that $\Pi_P(\gamma): G \to G$ is a translation on the left by some element of $G$. As a result, we obtain a representation 
$$\Pi_P: \mathcal{P}_v \to G.$$
\begin{df}[Full holonomy group] \label{def:full}
We shall denote $\overline{\Hol}^{\dyn}(P, \Phi_t, v) = \overline{\Pi(\mathcal{P}_v)}$ the \emph{full dynamical holonomy group} of $\Phi$ at $v$, also called \emph{transitivity group} of the flow $\Phi_t$.
\end{df}
 Note that changing the base point $v$ or the identification $P_v \cong G$ defines another representation which is conjugate to the first one.

\begin{pps}[Holonomy bundle] \label{pps_hol_bdle}
Let $u \in P_v$ and denote $H = \overline{\Hol}^{\dyn}(P, \Phi_t, v)$. There exists a unique continuous, $\Phi_t$-invariant, principal $H$-subbundle $Q(u)$ of $P$ containing $u$.
\end{pps}
\begin{proof}
Fix some base point $u \in P, v = \pi(u)$. The associated holonomy bundle is 
$$Q(u) = \overline{\{\Pi_\gamma u, \gamma \in \mathcal{P}_{v, w}, w \in N\}} \subset P.$$
In particular, $Q(u)$ is $\Phi_t = \Pi^c$-invariant. The existence of paths in $\mathcal{P}_{v, w}$ for any $w$ tells us that the projection $\pi: P \to N$ restricts to a surjection $Q(u) \to N$. By definition, the fiber $Q_v(u) = \overline{\Hol}^{\dyn}(P, \Phi_t, v)$ is acted on simply transitively by $H$, and conjugacy by any holonomy $\Pi_\gamma \in \Pi(\mathcal{P}_{v, w})$ gives an $H$-equivariant homeomorphism between the fibers $Q_v(u) \cong Q_w(u)$. 

Finally, let $(U, \psi)$ be a local chart of $N$ around $v$. By the local product structure theorem (and shrinking $U$ if necessary) we obtain a continuous family $\gamma_{v, w}$ of stable-center-unstable paths from $v$ to $w$ for any $w \in U$. Then the continuity of holonomies with respects to the base point tells us that $(x, h) \in U \times H \mapsto \Pi_{v \to \psi(x)} u \cdot h$ defines a continuous local trivialization of $Q(u)$ at $v$.

This shows the existence of $Q(u)$, and the uniqueness comes from the holonomy invariance, which determines every fiber.
\end{proof}

Similarly, one defines a \emph{restricted full holonomy group} $\overline{\Hol}_0^{\dyn}(P, \Phi_t, v)$ by considering only paths in $\mathcal{P}_v$ which are homotopic to a point. This is a subgroup of $\overline{\Hol}^{\dyn}(P, \Phi_t, v)$. To investigate its nature, we need the following technical lemma:
\begin{lem} \label{lem_exist_csu}
Let $v, w \in N$.
\begin{enumerate}
\item If $\gamma: [0, 1] \to N$ is any continuous path, then $\gamma$ is homotopic to a path in $\mathcal{P}_{v, w}$. 
\item Let $\gamma_0, \gamma_1$ be two paths in $\mathcal{P}_{v, w}$ which are homotopic. Then, the homotopy $\gamma_t$ can be chosen so that $\gamma_t \in \mathcal{P}_{v, w}$ for every $t \in [0, 1]$.
\end{enumerate} 
\end{lem}
\begin{proof}
Let $\rho > 0$ be the radius of a ball in which the local product structure theorem can be applied at any point of $N$ (in particular, we assume that $\rho$ is less than the injectivity radius of $N$). The proof of (2) is in two steps: first, we show that any path in $\mathcal{P}_{v, w}$ can be chosen homotopic to a path in $\mathcal{P}_{v, w}$ which is concatenation of csu paths of diameter less than $\rho$; the same procedure applied to any path will give us (1). Then, we show that two such paths, when they are homotopic, can be chosen to be homotopic by a homotopy in $\mathcal{P}_{v, w}$ of the same kind.

Denote $\mathcal{P}_{v, w}^\rho$ the set of paths in $\mathcal{P}_{v, w}$ which are concatenations of paths of the form $\alpha = \alpha^s \circ \alpha^c \circ \alpha^u$ with $\alpha^s \subset W^s(\alpha(0)), \alpha^u \subset W^u(\alpha(1))$ and $\alpha^c$ a flowline, with $\alpha$ included in some subset of diameter less than $\rho$. Note that by the local product structure theorem, $\alpha$ is uniquely determined by its endpoints and the function $\alpha(0), \alpha(1) \mapsto \alpha$ is continuous on every set of diameter less than $\rho$.

\smallskip

We first show that any path $\gamma \in \mathcal{P}_{v, w}$ is homotopic in $\mathcal{P}_{v, w}$ to a path in $\mathcal{P}_{v, w}^\rho$. To do this, we may assume that $\gamma$ itself in contained in a ball $B$ of diameter less than $\rho$ by decomposing $\gamma$ into shorter paths. Then, consider the homotopy $\gamma_t$ defined by $\gamma_t(s) = \gamma(s)$ for $t \leq s \leq 1$ and $\gamma_t(s) = \alpha_t(s/t)$ where $\alpha_t$ is the unique csu path in $B$ with endpoints $\alpha_t(0) = \gamma(0), \alpha_t(1) = \gamma(t)$. The local product structure theorem tells us that $\gamma_t$ defined a homotopy, and by construction $\gamma_t \in \mathcal{P}_{v, w}$ and $\gamma_1 \in \mathcal{P}_{v, w}^\rho$.

As announced, notice that if $\gamma$ is any continuous path, then the same procedure applies and gives a homotopy between $\gamma$ and an element of $\mathcal{P}_{v, w}^\rho$.

\smallskip

We now show that if $\gamma_0, \gamma_1 \in \mathcal{P}_{v, w}^\rho$ are homotopic with homotopy $\gamma_t$, then there exists a homotopy $\alpha_t$ between $\gamma_0, \gamma_1$ such that $\alpha_t \in \mathcal{P}_{v, w}^\rho \forall t$. By uniform continuity of the homotopy $\gamma$, we may cut up our homotopy : $[0, 1]^2 \to N$ into small squares of side length $1 / m = \delta > 0$ such that the image of the restriction of the homotopy to each square $[i/m, (i+1)/m] \times [j/m, (j+1)/m]$ is covered by a ball of diameter less than $\rho$, giving us $m^2$ small homotopies. Let $\gamma_t^{i, j}(s) = \gamma_{(i+t)/m}((j+s)/m)$. Then, the previous construction gives us a homotopy between $\gamma^{i, j}_t$ and some path $\alpha^{i, j}_t \in \mathcal{P}_{\gamma_t^{i, j}(0), \gamma_t^{i, j}(1)}^\rho$. Moreover, the local product structure theorem tells us that the maps $s, t \mapsto \alpha_t^{i, j}(s)$ are continuous on their domains. As a result, patching together these $\alpha^{i, j}_t$ gives a homotopy $\alpha_t$ in $\mathcal{P}_{v, w}^\rho$ between $\gamma_0$ and $\gamma_1$: indeed, the local product structure theorem also applies at the junctions of the squares to ensure that the homotopy is continuous.
\end{proof}

As a result, the classical properties of the restricted holonomy group hold:
\begin{csq} \label{cor_restrict_hol}
The restricted dynamical holonomy group $\overline{\Hol}_0^{\dyn}(P, \Phi_t, v)$ is the connected component of $\overline{\Hol}^{\dyn}(P, \Phi_t, v)$. Thus, if $q:\tilde{N} \to N$ denotes the universal cover, $\widetilde{\Phi_t}$ is the lift of $\Phi_t$ to $\widetilde{N}$ and $\tilde{v}$ is a lift of $v$ to $\tilde{N}$, then the holonomy group $\overline{\Hol}^{\dyn}(q^*P, \widetilde{\Phi_t}, \tilde{v})$ defined by the lifts of holonomies of $N$ is connected.
\end{csq}
\begin{proof}
The lemma tells us that the group generated by elementary holonomies along null homotopic loops is path connected, so its closure $\overline{\Hol}_0^{\dyn}(P, \Phi_t, v)$ is connected; in particular it is a subset of the identity component of $\overline{\Hol}^{\dyn}(P, \Phi_t, v)$. Stating this result at the level of the universal cover, we obtain that $\mathcal{P}_{\tilde{v}}$ is path connected, hence $\overline{\Hol}_0^{\dyn}(q^*P, \widetilde{\Phi_t}, \tilde{v})$ is connected.

Then, to show that $\overline{\Hol}_0^{\dyn}(P, \Phi_t, v)$ is exactly the identity component, observe that as in the classical case there is a natural map $\mathcal{P}_v \to \pi_1(M)$. The kernel of this morphism is the set $\mathcal{P}_{v, 0}$ of csu paths which are homotopic to a point. Hence, $\overline{\Hol}^{\dyn} / \overline{\Hol}_0^{\dyn}$ is discrete, which implies that $\overline{\Hol}_0^{\dyn}$ contains the identity component of $\overline{\Hol}^{\dyn}$; this concludes the proof.
\end{proof}

\begin{rmq} \label{rmq_univ_cover}
For applications to the geodesic flow on a manifold $M$, we consider $N = SM$ where $SM$ is the unit tangent bundle of $M$. Do note that when $\dim M = n \geq 3$, there is a natural map $S\widetilde{M} \to SM$ which is an universal cover. Indeed, the fiber of $S\widetilde{M}$ is a sphere of dimension $n - 1$, hence is simply connected; this implies that $S\widetilde{M}$ is simply connected by the homotopy exact sequence. In particular, the flow on $\widetilde{SM}$ obtained by lifting the geodesic flow of $M$ is identified with the geodesic flow of $\widetilde{M}$.
\end{rmq}

We end this section with a useful characterization of the ergodicity of the flow $\Phi_t$ in terms of the holonomy group. 
\begin{thm}[{{\cite{lef2023isometric}}}] \label{thm_ergod_frame}
Assume that $G$ is compact, $\varphi_t$ preserves some smooth volume form $\mu$ and $\Phi_t$ is isometric. The group $H = \overline{\Hol}^{\dyn}(P)$ is equal to $G$ if and only if $\Phi_t$ is ergodic with respects to $\mu$.
\end{thm}

\begin{ex}[The geodesic flow on real hyperbolic manifolds, continued] \label{ex:hyp_geod_group}

We have seen that dynamical holonomies on the stable/unstable bundles of real hyperbolic manifolds are scalar multiples of Levi-Civita holonomies on $SM$ in Example \ref{ex:hyp_geod_hol}. This shows that the associated $H^{E_u} = \overline{\Hol}^{\dyn}(d\phi_t)$ preserves the conformal class of the Sasaki metric. 

However, keeping track of the multiplicative factor is difficult; it is more convenient to study the image $PH^{E_u}$ of this group in the projective general linear group $\PGL(E_u)$. 

On one hand, it is known that the Levi-Civita holonomy group of the real hyperbolic space $\R\HH^n= \SO(n) \backslash \SO(1, n)$ is isomorphic to $\SO(n)$ (this follows from the classical theory of symmetric spaces, see for instance \cite[Proposition 10.79]{besse1987einstein}). As a result, if $M$ is a discrete quotient of $\R \mathbb{H}^n$, then its Levi-Civita holonomy group is equal to $\SO(n)$ or $\Or(n)$ (depending on whether $M$ is orientable or not). Then, holonomies on $SM$ for the pullback connection $p^*\nabla^{LC}$ (with $p: SM \to M$ the projection) can be seen as restrictions of holonomies on $M$ to hyperplanes; hence the holonomy group of $p^*\nabla^{LC}$ is equal to $\SO(n-1)$.

On the other hand, it follows from Lemma \ref{lem_exist_csu} that central/stable/unstable paths are dense in each homotopy class. As a result, the identity component $PH^{E_u}_0$ is isomorphic to $\PSO(n-1)$.
\end{ex}

\subsection{A connection for isometric extensions} \label{sec:hol_kanai}

Following \cite{kanai1993differential}, let $G$ be a compact Lie group and $(P, \Phi_t)$ be a principal $G$-extension of some Anosov flow $\varphi_t$ on a closed manifold $N$. Assume that there exists an auxiliary connection $D$ on $P$ and an adapted right-invariant metric $g^P$ such that $\Phi_t$ acts isometrically between fibers. Then, the flow $\Phi_t$ is partially hyperbolic and there is a decomposition 
$$TP = E_s^P \oplus E_c^P \oplus E_u^P$$
where $E_{s/u}^P$ are the stable/unstable bundles of the flow $\Phi_t$ and $E_c^P$ is the central bundle. The latter splits as $E_c^P = \R X^P \oplus \mathbb{V}$ where $X^P$ is the generator of $\Phi_t$ and $\mathbb{V} = \ker d\pi$ is the vertical bundle. Moreover, the projection $d\pi: TP \to TN$ projects $X^P$ on $X^N$ and $E_{s/u}^P$ onto $E_{s/u}^N$.

As a result, there is an obvious choice of connection on $P$ such that holonomies along stable/unstable/central paths are flat: the horizontal space should be 
$$\mathbb{H} = E_s^P \oplus E_u^P \oplus \R X^P.$$
We shall denote $\Theta$ the associated connection form, which is defined by $\Theta(Y) = 0$ for $Y$ a section of $\HH$ and $\Theta(A^*) = A$ where $A \in \frk{g}$ and $A^* = \partial_t \vert_{t = 0} \exp(tA)$ is the associated \emph{fundamental vector field}; note that fundamental vector fields generate the vertical bundle $\mathbb{V}$ so that $\Theta$ is uniquely defined.

\begin{lem} \label{lem_dyn_connection} Let $\Phi_t: P \to P$ be any partially hyperbolic principal extension of some Anosov flow $\varphi_t: N \to N$. The form $\Theta$ on $P$ defines a continuous $G$-principal connection on $P$, which is $\Phi_t$-invariant.
\end{lem}

\begin{proof}
Let $a \in G$ and denote $R_a$ the right action of $a$ on $P$, we have to show that $(R_a)_* \mathbb{H} \subset \mathbb{H}$. We have, as $\Phi_t$ is a principal extension:
$$\Phi_t \circ R_a = R_a \circ \Phi_t.$$
It is now clear that $R_a$ preserves $X^P$, and also that $R_a$ will also preserve the set of tangent vectors which are expanding/contracting under the action of $\Phi_t$, i.e. $E_{s/u}^P$. This shows that $(R_a)_* \mathbb{H} \subset \mathbb{H}$. 

The fact that $\mathbb{H}$ is $\Phi_t$-invariant is clear from the definition of $\mathbb{H}$. Let $A \in \frk{g}$ and $A^*$ be the associated fundamental vector field. If $f$ is a smooth function on $P$, then
$$\Phi_{t*} A^* f = \left . \frac{d}{ds} \right \vert_{s = 0} f \circ \Phi_t \circ R_{\exp(sA)} = \left . \frac{d}{ds} \right \vert_{s = 0} f \circ R_{\exp(sA)} \circ \Phi_t = A^* f \circ \Phi_t.$$
As a result, 
$$\Phi_t^* \Theta(A^*)(v) = \Theta_{\varphi_t(v)}(\Phi_{t*} A^*(v)) = \Theta(A^*) (\Phi_t(v))$$
or in other words $\Phi_t^* \Theta = \Theta$.
\end{proof}

\begin{rmq}
Alternatively, one can obtain the existence (and uniqueness) of such a flow-invariant connection using a fixed point argument, see \cite[Proposition 4.2]{kanai1993differential}.
\end{rmq}

One can then define the parallel transport on $P$ with respect to $\Theta$. Indeed, parallel transport is the solution to some \emph{linear} ordinary differential equation with continuous coefficients, so is well-defined. 

\begin{df}[Classical holonomy group] \label{def:classical}
Let $v \in N$ be a base point and identify the fiber $P_v$ with $G$. We denote $\Hol^{\dyn}(P, \Phi_t, v) \subset G$ the holonomy group associated to the connection form $\Theta$.
\end{df}

\begin{rmq} \label{rmq_classical}
If the connection form $\Theta$ is $\mathcal{C}^1$, then the classical theory for connections applies (see in particular \cite[Vol 1, Ch 1, Thm 4.2, Thm 7.1]{kobayashi1996foundations}). The identity component of the classical holonomy group $D = \Hol^{\dyn}(P, \Phi_t, v)$ is a closed subgroup of $G$. Moreover, $D$ has at most countably many components. Finally, there exists holonomy bundles $Q \subset P$ which are reductions of the structure group $G$ of $P$ to $D$. Such holonomy bundles are invariant under holonomy, hence under the action of $\Phi_t$ by construction; hence they define a reduction of $(P, \Phi_t)$ (see Section \ref{sec:ext_flow}).
\end{rmq}

It remains to see that these holonomies are the same as the ones defined in the previous section.
\begin{lem}
Let $v \in N$, $w \in W^s(v)$ close to $v$. Then the stable holonomy $\Pi^s_{v \to w}$ is the same as the parallel transport along any path from $v$ to $w$ contained in $W^s(v)$ with respect to $\Theta$. 
\end{lem}
\begin{proof}
Let $u'$ be the parallel transport of $u$ with respects to $\Theta$. By definition, any path in $W^s(v)$ lifts horizontally to a path in $W^s(u)$. As a result, $u' \in W^s(u)$, or equivalently:
$$d(\Phi_t(u), \Phi_t(u')) \to 0.$$
Recall that we fixed a connection $D$ such that the distance $g^P$ on $P$ is adapted to $D$. As a result, when $t$ is large enough, we may trivialize $P$ on some open set containing the unique geodesic from $\varphi_t(v)$ to $\varphi_t(w)$ so that the distance writes
\begin{align*}
d(\Phi_t(u), \Phi_t(u'))^2 &= d(\Phi_t(u), \tau_{\varphi_t(v) \to \varphi_t(w)} \circ \Phi_t(u))^2 \\
&+ d(\tau_{\varphi_t(v) \to \varphi_t(w)} \circ \Phi_t(u), \Phi_t(u'))^2.
\end{align*}
As a result, $d(\tau_{\varphi_t(v) \to \varphi_t(w)} \circ \Phi_t(u), \Phi_t(u'))$ converges to zero, and as $\Phi_t$ acts isometrically between fibers, we thus obtain 
$$u' = \lim_{t \to + \infty} \Phi_{-t} \circ \tau^D_{\varphi_t(v) \to \varphi_t(w)} \circ \Phi_t u = \Pi^s_{v \to w} u.$$
\end{proof}

\begin{csq} \label{cor_inclusions}
Given a base point $v \in N$, we have a chain of dense inclusions of holonomy groups at $v$:
$$\Pi_P(\mathcal{P}_v) \subset \Hol^{\dyn}(P, \Phi_t^P, v) \subset \overline{\Hol}^{\dyn}(P, \Phi_t^P, v).$$
\end{csq}

\section{The curvature morphism} \label{sec:curvature}

\subsection{Definition and equivariance property}

In the following, $(P, \Phi)$ is a $G$-principal extension of a contact Anosov flow $\varphi_t: N \to N$ preserving a contact form $\lambda$ such that the dynamical connection form $\Theta$ on $P$ is $\mathcal{C}^1$.

In this context, there is also a well-defined curvature form $\Omega$ which is a $\mathfrak{g}$-valued 2-form on $P$. It is defined by the formula
$$\Omega(Y, Z) = d\Theta(Y, Z) + [\Theta(Y), \Theta(Z)].$$

Notice that $\Omega$ is continuous since $\Theta$ is $\mathcal{C}^1$. Moreover, it is well known that it is horizontal and $G$-equivariant for the $\Ad$-action on $\mathfrak{g}$, so it induces a 2-form on the associated adjoint bundle 
$$\omega \in \mathcal{C}(N, \Lambda^2 T^* N \otimes \Ad(P)).$$
Note that there are induced flows $\Psi_t = \Lambda^2 d\Phi_t \otimes \Phi_t$ on the tensor product $\Lambda^2 T^*P \otimes (\frk{g} \times P)$ over $P$ and $\psi_t = \Lambda^2 d\varphi_t \otimes \Ad \Phi_t$ on the tensor product $\Lambda^2 T^*N \otimes \Ad(P)$ over $N$. We will still denote $X$ their generators.

\begin{lem} \label{lem_curvature}
The curvatures $\Omega, \omega$ satisfy the following properties:
\begin{enumerate}
\item[(1)] $\iota_{X}\Omega = 0$ and $\iota_X \omega = 0$;
\item[(2)] $\Psi_t^* \Omega = \Omega$ and $\psi_t^* \omega = \omega$.
\end{enumerate}
\end{lem}
\begin{proof}
We first prove the two properties for $\Omega$.

Recall that $\Omega(Y, Z) = d\Theta(Y, Z) + [\Theta(Y), \Theta(Z)]$. We have already seen that $\mathscr{L}_{X} \Theta = 0$ in Lemma \ref{lem_dyn_connection}, and $\iota_X \Theta = 0$ by definition of $\Theta$. It follows that $\iota_{X} d\Theta = 0$ by Cartan's formula. 

Then, to obtain (2), observe that $\Psi_t^* \Theta = \Theta$ by Lemma \ref{lem_dyn_connection}, so
$$\Psi_t^* \Omega = \Psi_t^* (d\Theta + [\Theta, \Theta]) = d \Psi_t^* \Theta + [\Psi_t^* \Theta, \Psi_t^* \Theta] = \Omega.$$

We then prove the properties for $\omega$. The projection $\pi : P \to N$ induces an equivariant projection $q : \Lambda^2 T^*P \otimes \frk{g} \to \Lambda^2 T^*N \otimes \Ad(P)$ by
$$q(u, \pi_u^*\alpha \otimes A) = \alpha \otimes [u, A]_{\Ad(P)}, \quad q(u, \beta \otimes A) = 0$$
where $u \in P, \alpha \in \Lambda^2 T_{\pi(u)}^* N, \beta \in \Lambda^2 T_u^* P, A \in \frk{g}$ such that $\beta$ is vertical.

Then, $\Omega = q^* \omega$ and as $q$ is equivariant, $q^*$ preserves $X$ and $\mathscr{L}_X$. As a result, the two properties for $\omega$ follow from the ones for $\Omega$.

\end{proof}

The flow invariance can be used to further restrict the curvature form.
\begin{csq} \label{cor_restrict}
The restrictions of $\omega$ to $\Lambda^2 E_s^*, \Lambda^2 E_u^* \subset \Lambda^2 T^* N$ are zero.
\end{csq}
\begin{proof}
Consider $Y^-, Z^- $ two sections of $E_s$. Then, as $\omega$ is invariant under the product flow $\psi_t = \Lambda^2 d\varphi_t \otimes \Ad \Phi_t$, we have
$$\psi_t^* \omega(Y^-, Z^-) = \omega(\Phi_{t*} Y^-, \Phi_{t*} Z^-) = \omega(Y^-, Z^-).$$
However, $\Phi_t^* Y^-, \Phi_t^* Z^-$ converge exponentially fast to 0, so as $\omega$ is continuous and $N$ is compact, we have $\vert \omega(Y^-, Z^-) \vert \leq C e^{-ct}$ for some $C, c > 0$. As a result, $\omega(Y^-, Z^-) = 0$.

A similar argument applies to the $\Lambda^2 E_u^*$ part.
\end{proof}

Decompose $\Lambda^2 T^* N = \Lambda^2 E_s \oplus E_s^* \otimes E_u^* \oplus \Lambda^2 E_u \oplus \R \lambda \otimes E_s^* \oplus \R \lambda \otimes E_u^*$. By item (1) of Lemma \ref{lem_curvature}, we see that $\omega$ vanishes on $\R \lambda \otimes E_s^*, \R \lambda \otimes E_u^*$; and by Corollary \ref{cor_restrict}, $\omega$ also vanishes on $\Lambda^2 E_s, \Lambda^2 E_u$. Hence, $\omega$ is determined by its values on $E_s^* \otimes E_u^*$.

\smallskip

We now introduce another way of writing $\omega$ which will be useful in the following (see Section \ref{sec:hyp_frame} for motivation). 

The fact that the flow-invariant contact form $\lambda$ induces a duality $d\lambda : E_s^* \cong E_u$ allows us to identify $E_s^* \otimes E_u^*$ with $\End(E_u)$. Under this identification, the flow induced by $d\phi_t$ on $E_s^* \otimes E_u^*$ corresponds to the flow $\Ad(d\varphi_t): A \mapsto d\phi_{-t} \cdot A \cdot d\phi_t$ on $\End(E_u)$.

On the other hand, the choice of a bi-invariant metric on the Lie algebra $\frk{g}$ (which exists as $G$ is compact) determines a metric $g_{\Ad}$ on $\Ad(P) = P \times_{\Ad} \frk{g}$ which is preserved by the flow and allows us to identify $\Ad(P)$ with its dual $\Ad(P)^*$. This identification exchanges the flows induced by $\Ad(\Phi_t)$ and its transpose $\Ad(\Phi_t)^T$. 

The above reasoning lets us see $\omega$ as a section of $\Ad(P)^* \otimes \End(E_u)$, or equivalently as a bundle map which we will denote $F = F^P: \Ad(P) \to \End(E_u)$ defined by
$$g_{\Ad}(\omega^P(Y^+, Y^-), \xi) = d\lambda(F^P(\xi)Y^+, Y^-) \label{eq:identifications}$$
for sections $\xi$ of $\Ad(P)$, $Y^+$ of $E_u$ and $Y^-$ of $E_s$.

Moreover, the fact that $\omega$ is $\psi_t$-invariant (item (2) of Lemma \ref{lem_curvature}) translates into the following:
$$F^P \circ \Ad(\Phi_{t})^T = \Ad(d\varphi_t) \circ F^P$$
where $\Ad(d\varphi_t)$ is the naturally induced flow on $\End(E_u)$, which is just conjugacy by $d\varphi_t$; and $\Ad(\Phi_t)$ is the associated flow on the associated bundle $\Ad(P)$ (see Section \ref{sec:ext_flow}) which in this case is the pullback by $\Phi_t$ of two-forms.

The flow-equivariance of $F$ can then be strengthened to some form of equivariance for the holonomy groups. To simplify the notations, we will simply denote by $d\varphi_t$ and $\Phi_t$ the corresponding induced actions on associated bundles.

\begin{pps} \label{pps_equivariance_master}
Assume that $G$ is compact, $d\varphi_t$ is fiber bunched with constant $\beta$ and that $F$ is $\alpha$-Hölder continuous with $\beta < \alpha$. Then the map $F$ is equivariant for the representations of the path monoid:
$$\forall \gamma \in \mathcal{P}_v, \ \Pi_{\End(E_u)}(\gamma) \circ F = F \circ \Pi_{\Ad(P)}(\gamma)^{T}.$$
\end{pps}
\begin{proof}
Let $v \in N$, $\varepsilon > 0$ small enough, $V_\varepsilon(v)$ be the neighborhood constructed in the Theorem \ref{thm_stable_hol}. For $w \in V_\varepsilon(v)$, denote $h_{t} = \Pi^s_{\varphi_t (v) \to \varphi_t(w)}$, $\tau_{t} = \tau_{\varphi_t(v) \to \varphi_t(w)}$ on each relevant bundle. Consider 
$$S(t, \xi) = (h_{v \to w} F)(\xi)(\varphi_t(v)) - F(\xi)(\varphi_t(w)).$$
We want to show that $S(0, \xi) = 0$. Note that as stable holonomies commute with the flow and $F$ is flow-equivariant we obtain 
$$S(0, \xi) = d\varphi_{-t} S(t, \Phi_t(\xi)^T).$$
Since $\Phi_t^T$ acts isometrically on $\Ad(P)$ and $N$ is compact,
$$\Vert S(0, \xi) \Vert \leq C \Vert d\varphi_{-t} \vert_{\End(E_u)} \Vert \cdot \Vert h_{t} F - F \Vert \cdot \Vert \xi \Vert.$$
On one hand, as $d\varphi_t$ is fiber-bunched, we have $\Vert d\varphi_{-t} \Vert \leq C \nu^{-t\beta}$. On the other hand, the Hölder continuity of $F$ and the bound \eqref{eq:unif_bound_hol_tp} imply that 
\begin{align*}
\Vert h_{t} F - F \Vert &\leq \Vert h_{t} F - \tau_{t} F \Vert + \Vert \tau_{t} F - F \Vert \\
&\leq C d_s(v_t, w_t)^\alpha.
\end{align*}
Using the exponential convergence \eqref{eq:shrink_distance}, we obtain
$$\Vert S(0, \xi) \Vert \leq C \nu(v)^{-t\beta + t\alpha(1 - \varepsilon)} \Vert \xi \Vert.$$
As a result, $S(0, \xi) = 0$ as $t \to +\infty$ for any $\xi$ whenever $\alpha (1 - \varepsilon) \leq \beta$, which happens when $\varepsilon < 1 - \frac{\beta}{\alpha}$; we may choose a suitable $\varepsilon$ whenever $\alpha > \beta$.

This shows that $F$ is equivariant under stable holonomies. Similarly, $F$ is equivariant under unstable holonomies, and we already know that $F$ is equivariant under central holonomies, hence the result.
\end{proof}

\begin{csq} \label{cor_eq_curvature}
Assume that $M$ is a (closed negatively curved) $\delta$-pinched manifold, $\delta > 1/4$. If the dynamical connection form $\Theta$ of its frame flow is $\mathcal{C}^{1/\sqrt{\delta}}$, then its associated curvature map $F$ is well-defined and equivariant for the representations of $\mathcal{P}_v$.

In particular, this is always the case if $M$ is strictly $1/4$-pinched and $\Theta$ is $\mathcal{C}^{1, 1}$.
\end{csq}

\begin{csq} \label{cor_rank}
Under the assumptions of Proposition \ref{pps_equivariance_master}, the bundle map $F$ has constant rank.
\end{csq}
\begin{proof}
If $v, w \in N$, let $\gamma \in \mathcal{P}_{v, w}$ be any path, which exists by Lemma \ref{lem_exist_csu}. Then, Proposition \ref{pps_equivariance_master} yields
$$\Pi_{E_u}(\gamma) \circ F_v = F_w \circ \Pi_P(\gamma)$$
which implies that $F_v$ and $F_w$ have the same rank.
\end{proof}

\begin{csq} \label{cor_ambrose_singer}
Choosing an identification of a fiber $\Ad(P)_v \cong \frk{g}$, the subspace $\ker F_v^\perp$ corresponds to the Lie algebra $\frk{d}^P$ of the \emph{classical} holonomy group $D^P = \Hol^{\dyn}(P, \Phi_t, v)$. This identification is compatible with the holonomy action of $H^P = \overline{\Hol}^{\dyn}(P, \Phi_t, v)$.
\end{csq}
\begin{proof}
The Ambrose-Singer holonomy theorem \cite[Vol. 1, Ch. 2, Theorem 8.1]{kobayashi1996foundations} writes as follows. The Lie algebra $\frk{d}^P$ of the \emph{classical} holonomy group $D^P = \Hol^{\dyn}(P)$ is generated by the values taken by the curvature $\Omega$ on the holonomy bundle. 

Notice that as $\Omega$ is equivariant with respect to the action of $H^P = \overline{\Hol}^{\dyn}(P, \Phi_t, v)$ by Proposition \ref{pps_equivariance_master}, the same is true for the action of $D^P \subset H^P$ (Corollary \ref{cor_inclusions}). This implies that $\Omega$ is parallel with respect to the dynamical connection form $\Theta$.

As a result, the values taken by $\Omega$ along the holonomy bundle associated to $\Theta$ correspond, through parallel transport, to the values taken at a single point of the holonomy bundle. More precisely, the Lie algebra $\frk{d}^P \subset \frk{g}$ is generated by the values taken by $\Omega_u$ for any $u \in P_v$. On the base, we obtain an identification between the values taken by the 2-form $\omega$ at $v$ and $\frk{d}^P \subset \frk{g} \cong \Ad(P)_v$.

Then, $F$ is obtained from $\omega$ by taking a dual with respect to a metric on $\Ad(P)$, which in particular implies that $(\ker F_v)^\perp$ coincides with the range of $\omega$.

As explained above, the maps $F, \omega, \Omega$ are parallel with respect to the holonomy action of $H^P$ and not just $D^P$, which implies that the identification $\frk{d}^P \cong (\ker F_v) ^\perp$ commutes with the action of $H^P$.
\end{proof}

\subsection{Examples}

We now give three important examples of explicit com-putations of the curvature $\omega$, in which the connection form $\Theta$ happens to be smooth.

\subsubsection{Flat flows} \label{sec:flat}

Let $N$ be a closed manifold and $\varphi_t$ an Anosov flow. Then, the "trivial" extension defined on the principal $\{\identity\}$-bundle $\{\identity\} \times N$ by the formula $\Phi_t(x, \identity) = (\varphi_t(x), \identity)$ is clearly an isometric extension of $\phi_t$ which admits (trivial) dynamical holonomies, with a trivial holonomy group and zero curvature. This generalizes as follows.

\begin{df}[Flat flows] \label{def:flat}
Let $N$ be a closed manifold and $\varphi_t$ a flow. Denote $\widetilde{N}$ the universal cover of $N$. The vector field defining $\varphi_t$ lifts uniquely to $\widetilde{N}$ and defines a flow $\tilde{\varphi}_t$. For any Lie group $G$, consider the flow $\widetilde{\Phi}_t$ acting on $\widetilde{N} \times G$ defined by 
$$\widetilde{\Phi}_t(x, a) = (\tilde{\varphi}_t(x), a).$$
Any representation $\rho: \pi_1(M) \to G$ acts on $\widetilde{N} \times G$ by $\gamma \cdot (\tilde{v}, a) = (\gamma \cdot \tilde{v}, \rho(\gamma)^{-1}a)$. The flow $\widetilde{\Phi}_t$ then descends to a flow $\Phi_t$ on the flat bundle $P = \widetilde{N} \times_\rho G \coloneq (\widetilde{M} \times G) / \pi_1(M)$, which we call the \emph{flat flow} associated to $\rho$. Seeing $\widetilde{N} \to N$ as a principal $\pi_1(M)$ bundle, we can interpret $\Phi_t$ as the induced flow on the bundle associated to the representation $\rho$.
\end{df}

\begin{lem} \label{lem_flat_flows}
Let $(P, \Phi_t)$ be a smooth principal extension of $\varphi_t$ with structure group $G$. Assume that the associated dynamical connection form $\Theta$ is $\mathcal{C}^1$. The following are equivalent:
\begin{itemize}
\item The dynamical horizontal distribution $\mathbb{H}$ is integrable;
\item The dynamical curvature $\omega$ vanishes;
\item The flow $\Phi_t$ is flat.
\end{itemize}
\end{lem}
\begin{proof}
Assume first that $\omega = 0$. Denote $q: \widetilde{N} \to N$ the universal cover. Then, $\Theta$ lifts to a flat connection form $\widetilde{\Theta}$ on $q^* P$, which implies that $q^*P\to \widetilde{N} \cong \widetilde{N} \times G$ is a trivial bundle. Denote $\widetilde{\Phi}_t$ the lift of $\Phi_t$ to $q^* P$, then by construction the action of $\widetilde{\Phi}_t$ is by holonomy along paths for $\widetilde{\Theta}$, hence in our above trivialization of $q^*P \to \widetilde{N}$ the action of $\widetilde{\Phi}_t$ writes 
$$\widetilde{\Phi}_t(v, u) = (\tilde{\varphi}_t(v), u) \ \forall (v, u) \in \widetilde{N} \times G.$$

Conversely, assume that $\Phi_t$ is flat and denote $\widetilde{\Phi}_t$ the lift of $\Phi_t$ to $\widetilde{N}$, so that in a trivialization $q^*P \cong \widetilde{N} \times G$ we have 
$$\widetilde{\Phi}_t(v, u) = (\varphi_t(v), u) \ \forall (v, u) \in \widetilde{N} \times G.$$
Then, for any $u \in G$, the submanifold $S_u = \widetilde{N} \times \{u \}$ is obviously $\widetilde{\Phi}_t$-invariant. The action of $\widetilde{\Phi}_t$ on $S$ is the same as that of $\varphi_t$ on $\widetilde{N}$, which lets us identify the corresponding stable and unstable subspaces. We thus see that the tangent distributions to the $S_u, u \in G$ form the dynamical horizontal distribution of $\widetilde{\Phi}_t$; this distribution is integrable hence the associated curvature $\omega$ vanishes.
\end{proof}

Note that flat extensions have discrete $\Hol^{\dyn}$, but may have nondiscrete $\overline{\Hol}^{\dyn}$. It can even happen that $\overline{\Hol}^{\dyn} = G$: one can obtain examples by constructing representations of the fundamental group $\pi_1(N)$ into $G$ with dense image.

\subsubsection{The frame flow on the real hyperbolic space} \label{sec:hyp_frame}

Let $M = \mathbb{H}^n(\R) = \SO_0(1, n) / \SO(n)$ be the real hyperbolic space of dimension $n$. By definition, $\SO_0(1, n)$ is a principal $\SO(n)$-bundle over $M$. We can naturally identify the frame bundle $FM$ to $\SO_0(n, 1)$ : this can be understood as follows. A general matrix in $\SO_0(n, 1)$ can be written as
$$\begin{pmatrix}
0 & u^T \\
u & A
\end{pmatrix}$$
where $u \in \R^{n}$ and $A \in \SO(n)$. This gives a splitting $\mathfrak{g} = \mathfrak{so}(1, n) = \mathfrak{so}(n) \oplus \R^n = \mathfrak{h} \oplus \mathfrak{m}$ where $\mathfrak{m}$ is $\ad(\mathfrak{h})$-invariant, as is typical of symmetric spaces. 

Note that $\SO_0(1, n)$ acts on itself by left translation. This allows us to identify the tangent space at some point $u \in \SO_0(1, n)$ with the tangent space at the identity, which is also the Lie algebra $\mathfrak{g}$. The subgroup $\SO(n)$ acts by left and right translations on $\SO_0(1, n)$, and the induced $\Ad$-actions on $\mathfrak{g}$ preserve $\mathfrak{m}$. 

Remark that the projection $p : \SO_0(1, n) \to M$ sends $\mathfrak{m}$ isomorphically on $T_IM$. To link $TM$ with $\SO_0(1, n)$, we are thus naturally led to consider the associated bundle $\SO_0(1, n) \times_{\Ad \mathfrak{h}} \mathfrak{m}$. The projection $p_*$ has kernel $\mathfrak{h}$, which implies that for $g \in \SO_0(1, n), h \in \SO(n), m \in \mathfrak{m}$, 
$$p_* (gh \ad(h^{-1})m) = p_* (gmh) = p_* (gm).$$
This shows that $p_*$ induces an isomorphism $\SO_0(1, n) \times_{\Ad \mathfrak{h}} \mathfrak{m} \cong TM$. 

Then, to recover the entire frame bundle $FM$, we recall that the fibers $F_xM$ are the isometries $u: \R^n \to T_xM$. Note that $\SO_0(1, n)$ acts by isometries on the fibers of $\SO_0(1, n) \times_{\Ad \mathfrak{h}} \mathfrak{m}$, and a dimension count shows that every isometry is obtained this way. As a result, $FM$ is isomorphic to $\SO_0(1, n)$.

As a result, we can see $SM$ as $\SO_0(1, n) / \SO(n-1)$ where $\SO(n - 1) \leq \SO(n)$ is identified to the diagonal subgroup $diag(1, U)$.

Note that computations in bundles over $M, SM$ can be carried out only at the tangent space at identity, as every element introduced thus far is right-invariant. One can thus read the different elements introduced so far in a matrix form:
$$\begin{pmatrix}
0 & v & u^T \\
v & 0 & - w^T \\
u & w & A
\end{pmatrix}.$$
Here, $A \in \SO(n - 1)$ represents the vertical bundle $\mathbb{V} \to SM$. Modding out the bottom right block which represents $\SO(n)$, we recover $T_I M$ with the first basis vector $v = u(e_1)$, which we take as the direction of a geodesic. The normal bundle $v^\perp$ is then represented by $u$.

Modding out $A \in \SO(n - 2)$ instead, we obtain $T_I SM$ with the direction $X = E_{1, 2} + E_{2, 1}$ of the geodesic flow and the vertical/horizontal distributions $w, u$. The stable and unstable bundle are then represented by linear combinations of $u, w$ vectors. More precisely, a basis $U_2^+, \dots, U_n^+$ of $E_u$ is easily seen to be 
$$U_i^+ = - E_{1, i + 1} - E_{2, i + 1} - E_{i + 1, 1} + E_{i + 1, 2}.$$
Indeed, the geodesic flow acts by the exponential $\exp(tX)$. In the Lie algebra, this action becomes the bracket with $X$, and one easily computes $[X, U_j^+] = U_j^+$. Similarly, a basis $U_2^-, \dots, U_n^-$ of $E_s$ is
$$U_i^- = - E_{1, i + 1} + E_{2, i + 1} - E_{i + 1, 1} - E_{i + 1, 2}$$
with the relations $[X, U_j^-] = - U_j^-$.

Let then $R_{i, j} = E_{i+1, j+1} - E_{j+1, i+1}$ be the basis of $\mathfrak{so}(n - 2)$. We have the commutation relations $[U_i^+, U_j^+] = 0, [U_i^-, U_j^-] = 0, [U_i^+, U_i^-] = 2X, [U_i^+, U_j^-] = R_{i, j}$ for $i \neq j$. As a result, the curvature $\Omega$ satisfies for $i \neq j$:
$$\Omega(X, .) = 0, \quad \Omega(U_i^+, U_j^+) = 0 = \Omega(U_i^-, U_j^-),$$
$$\Omega(U_i^+, U_i^-) = 0, \quad \Omega(U_i^+, U_j^-) = R_{i, j}.$$

Then, if we identify $U_i^+ \in E_u$ with $e_i \in \mathcal{N}$ (see Section \ref{sec:background}), this computation shows that $F$ is the inclusion of the set $\Lambda^2 \R^{n-1}$ of skew-symmetric matrices in $\frk{gl}_{n-1}(\R)$.

\subsubsection{Homothetical curvature flows} \label{sec:exp_flow}

We now describe a way to deform flat flows in order to obtain new examples of extensions with smooth horizontal space. We begin with the case of extensions to line bundles (see \cite[pp 111-113]{cekic2024semiclassical} for more details on dynamical connections on line and circle bundles). Let $\varphi_t$ be an Anosov flow on a manifold $N$, $a \in \R$. We consider the extension $\Phi^a_t$ of $\varphi_t$ to $L = N \times \R$ defined by
$$\Phi_t(v, s) = (\varphi_t(v), s + ta) \ \forall v \in N, s \in \R.$$
Note that $\R_+^*$ acts by multiplication on $L$ and $\Phi^a_t$ commutes with this action; hence $(L, \Phi_t^a)$ is a $\R$ principal extension of $\varphi_t$.

There is a canonical choice of metric given by the unique (up to a scalar) translation invariant metric on $\R$ (which is just the standard metric on $\R$). Note that $\Phi^a_t$ acts by translation, hence preserves this metric. 

Remark also that the choice of a discrete subgroup $\Gamma = b \mathbb{Z} \subset \R$ allows us to construct various examples of isometric extensions of $\varphi_t$ to circle bundles, given by the $L / \Gamma$. 

A quick calculation shows:
\begin{lem} \label{lem_formula_central}
The dynamical connection form and curvature associated to $\Phi_t$ are given by 
$$\theta^a = \theta^0 + a \pi^* \lambda, \quad \omega^a = a d\lambda$$
where $\theta^0$ is the flat connection form on $L$ and $\pi: L \to N$ is the projection.
\end{lem}

Equivalently, $F: N \times \R \cong \Ad(L) \to \End(E_u) \cong N \times \R$ is the fiberwise homothety of factor $a$.

This construction may be twisted by representations of the fundamental groups just as in the case of flat flows. In fact, we have the following characterization:
\begin{lem} \label{lem_central_line}
Let $(P, \Psi_t)$ be an isometric, $G$ principal extension of $\varphi_t$ such that $G$ has dimension 1.

If the associated curvature form $\omega$ is given by $a d\lambda$ for some $a \in \R$, then, $(P, \Psi_t)$ can be obtained from the method described above. In particular, $P$ is a flat bundle.
\end{lem}

\begin{rmq}
Conversely, it follows from \cite[Lemma 5.1, Corollary 5.2]{hamenstadt1995invariant} that if $P$ is a trivial bundle, then the associated curvature form is exact and thus it is necessarily a constant multiple of $d\lambda$.
\end{rmq}

\begin{proof}
Let $\theta$ be the dynamical connection form associated to $\Psi_t$ and consider $\theta^0 = \theta - a \pi^* \lambda$. By construction, $\theta^0$ is a flat connection form, hence $P$ is flat. Moreover, define $\Phi_t^0$ the flow of parallel transport for $\theta^0$ along trajectories of $\varphi_t$; this is a flat flow by construction and $\Psi_t = \Phi_t^a$.
\end{proof}

\begin{ex}
Such a flow does appear in a geometrical context. Indeed, let $\varphi_t = \phi_t$ be the geodesic flow on a real hyperbolic manifold $M$, then the flow $\det d\phi_t$ acting on the density bundle of the unstable bundle $E_u \subset TSM$ of $\phi_t$ is an example of such a flow with curvature $d\lambda$, see \cite[Section 5]{hamenstadt1995invariant}. 
\end{ex}

We may rewrite the condition $\omega = a d\lambda$ as follows. The adjoint bundle $\Ad(P) \cong N \times \R$ is trivial. Then, in these coordinates, we can write 
$$F(v, x) = (v, ax \identity_{\End(E_u)\vert_v})$$
for $v \in N, x \in \R$. This motivates the following definition:
\begin{df} \label{def:central}
Let $(P, \Phi_t^P)$ be an isometric extension of $\varphi_t$ such that the associated curvature map $F$ is well-defined. We say that $\Phi_t^P$ has \emph{homothetical curvature} if $F$ is valued in the subbundle of homotheties of $\End(E_u)$.
\end{df}

\begin{rmq}
We notice that a similar, but stronger condition on the dynamical curvature has been introduced in \cite[Theorem 4.1.1]{cekic2024semiclassical} as an obstruction to the rapid mixing property for $\Phi_t$. Both conditions become equivalent when the center of $G$ has dimension $\leq 1$.
\end{rmq}

We can then generalize the previous lemma as follows:
\begin{lem} \label{lem_central_flat}
Let $(P, \Phi_t^P)$ be a homothetical curvature extension. 
\begin{enumerate}
\item[(1)] The classical holonomy Lie algebra $\frk{d}^P$ is central in the full holonomy Lie algebra $\frk{h}^P$ of $\Phi_t^P$;
\item[(1')] In particular, if $\Phi_t^P$ is ergodic and is not flat then the structure group $G$ of $P$ has a nondiscrete center.
\item[(2)] The bundle $P$ is always flat. 
\end{enumerate}
\end{lem}
\begin{proof}
We first prove (1). We know that the classical dynamical holonomy group $D^P = \Hol^{\dyn}(P, \Phi_t^p, v)$ is normal in its closure $H^P = \overline{\Hol}^{\dyn}(P, \Phi_t^P, v)$. For the Lie algebra $\frk{d}^P$ to be central in $\frk{h}^P$, it is thus enough to show that it is Abelian, as an Abelian ideal in a compact Lie algebra is central.

Let $q: \widetilde{N} \to N$ be the universal cover. The dynamical connection form $\Theta$ lifts to a connection form on $q^* P$, with holonomy group $D_0^P$ where $D^P$ is the classical holonomy group of $\Phi_t^P$. In particular, there exists a reduction of $(P, \Phi_t^P)$ to any holonomy bundle $Q$ associated to $q^* \Theta$ by Remark \ref{rmq_classical}. 

The Lie algebra $\frk{d}^P$ of $D^P$ is generated by the orthogonal of the kernel of $F^P$ by Corollary \ref{cor_ambrose_singer}. This implies that $\frk{d}^P$ has dimension $\leq 1$, hence is Abelian; this proves (1), and (1') follows from Theorem \ref{thm_ergod_frame}.

We now prove (2). If $\frk{d}^P$ is zero, then $\Theta$ is flat and we are done. Else, the structure group $D_0^P$ of $Q$ is isomorphic to $\R$ or $S^1$. Then, the connection 2-form $\omega$ associated to $q^* \Theta \vert_Q$ is a continuous, exact (as the structure group is Abelian) 2-form proportional to $d\lambda$ by assumption. Denote $\omega = a d\lambda$, then as $\omega$ is exact necessarily $a$ is constant. 

Following the proof of Lemma \ref{lem_central_line}, $\theta^0 = q^* \Theta - a q^* \pi^* \lambda \vert_Q$ defines a flat connection form on $Q$. This form naturally extends by right-invariance to $q^* P$ and is invariant under the action of the fundamental group $\pi_1(N)$ as the forms $q^* \Theta, q^* \pi^* \lambda$ as well as the right $G$ action are. We thus obtain a flat connection on $P$.
\end{proof}

\subsection{About the equivariance} \label{sec:equivariance}

We now present technical results that allow us to extract finer properties of the curvature map by exploiting its equivariance more effectively. To avoid confusion, we introduce a separate set of notations for this discussion.

We consider $V_1, V_2$ two finite dimensional vector spaces, $G_1, G_2$ two Lie subgroups of $\GL(V_1), \GL(V_2)$, $\mathcal{P}$ a monoid and $\rho_i: \mathcal{P} \to G_i$ two representations. We also denote $H_i$ the closure of the subgroups generated by $\rho_i(\mathcal{P})$ in $G_i$. Although it is not needed, one can assume that the images $\rho_i(\mathcal{P})$ are groups just as in the context of representations of path monoids. We study linear maps $f: V_1 \to V_2$ which are $\rho_1, \rho_2$-equivariant:
$$\forall \gamma \in \mathcal{P}, \rho_2(\gamma) \circ f = f \circ \rho_1(\gamma).$$
Obviously, we want some compactness assumption of $G_1$ to mimick the compactness of $G$ in the previous section.

\begin{pps} \label{pps_eq_orbits_top}
Assume that $G_1$ is compact. Then, for any $h_2 \in H_2$, there exists $h_1 \in H_1$ such that $f \circ h_1 = h_2 \circ f$.
\end{pps}
\begin{proof}
Let $h_2 \in H_2$, then by definition there exists a sequence $\gamma_k$ of elements of $\mathcal{P}$ such that $\rho_2(\gamma_k) \to h_2$. Then, as $G_1$ is compact, there is a subsequence of $\rho_1(\gamma_k)$ converging to some $h_1 \in H_1$. The identity $h_2 f = f h_1$ is then consequence of the continuity of $f$. 
\end{proof}

\begin{rmq} \label{rmq_compact_H}
The compactness of $H_1$ translates into some compactness property for $H_2$. For example, the $H_2$-orbits of $V_2$ which are in the image of $f$ are the continuous images of $H_1$-orbits, which are compact.
\end{rmq}

The previous result crucially relies on the compactness of $G_1$. We first have some purely algebraic consequences of this equivariance. 

\begin{pps}
\label{pps_eq_orbits_alg}

Assume that $f$ is a linear \emph{isomorphism}. If, for every $h_2 \in H_2$, there exists $h_1 \in H_1$ such that 
$$f \circ h_1 = h_2 \circ f,$$
%and conversely for every $g_1 \in G_1$, there exists $g_2 \in G_2$ such that
%$$F \circ \rho_1(g_1) = \rho_2(g_2) \circ F,$$
then there is an injective morphism of Lie groups $\chi: H_2 \to H_1$ such that $f$ is $H_2$-equivariant:
$$\forall h_2 \in H_2, \forall x \in V_1, \  h_2 \cdot f(x) = f(\chi(h_2) \cdot x).$$
\end{pps}

\begin{proof} Consider the mapping $\chi = \Ad f^{-1} : \End(V_2) \to \End(V_1)$ defined by $\chi(A) = f^{-1} \circ A \circ f$.
By assumption, the restriction $\chi \vert_{H_2}$ realizes a group morphism between $H_2$ and $H_1$, which is injective as a restriction. The fact that $f$ is equivariant with respect to this morphism is tautological.
\end{proof}

In general, one can assume that $f$ is an isomorphism onto its image by restricting and co-restricting it appropriately. Doing so will however force us to change $G_1, G_2$ and their subgroups by taking their images under the restriction maps.

\begin{pps} \label{pps_eq_orbits_repr}

Assume that $f$ is $(\rho_1, \rho_2)$-equivariant. Let $W_1 \subset V_1$ be an irreducible $H_1$-subrepresentation. Then either $W_1 \subset \ker f$ or $f : W_1 \to f(W_1)$ is a linear isomorphism and the image $f(W_1)$ is an irreducible $H_2$-subrepresentation.

In particular, if $V_1$ is $H_1$-irreducible, then $f$ is either 0 or injective, and in the latter case $\textrm{ran} f$ is an irreducible $H_2$-representation.
\end{pps}
\begin{proof} Remark first that $\ker f$ is $H_1$-invariant; indeed it is invariant under elements of $\rho_1(\mathcal{P})$ by assumption, and these are dense by definition of $H_1$. Consider now $W_1 \subset V_1$ a $H_1$-irreducible subrepresentation; then $\ker f \cap W_1$ is a subrepresentation of $W_1$, which implies that $W_1 \subset \ker f$ or $\ker f \cap W_1 = \{0\}$ by irreducibility. In this latter case, $f: W_1 \to W_2$ is an isomorphism and $W_2 = f(W_1)$ is stable under $H_2$ by Proposition \ref{pps_eq_orbits_top}. Further, if $W'_2 \subset W_2$ is a nonzero $H_2$-subrepresentation, then $W_1' = (f \vert_{W_1})^{-1}(W_2') \subset W_1$ is a $H_1$-subrepresentation, as elements of $\rho_1(\mathcal{P})$ are dense and preserve $W_1'$ by equivariance. Thus by irreducibility $W_1 = W_1'$ and $W_2 = f(W_1) = W_2'$, which proves that $W_2$ is irreducible.
\end{proof}

This is in contrast with the previous Proposition \ref{pps_eq_orbits_alg} where the roles of $V_1$ and $V_2$ are exchanged. More precisely, in the context of Proposition \ref{pps_eq_orbits_alg}, if $V_2$ is $H_2$-irreducible, then $V_1$ is $H_2$-irreducible as well, hence it is $H_1$-irreducible as $H_2 \hookrightarrow H_1$. But the previous result gives a converse: if $V_1$ is $H_1$-irreducible, then $V_2$ is $H_2$-irreducible as well.

\section{Proof of the main result} \label{sec:proof_kanai}

We are now ready to tackle the proof of Theorem \ref{thm_conj_Kanai}. The proof is done in three steps. First, we reduce the problem of Kanai to the following: under the assumptions of Theorem \ref{thm_conj_Kanai}, show that there exists a flow-invariant conformal structure on the unstable bundle $E_u$ associated to the geodesic flow (Lemma \ref{lem_reduction}). Then, we show that if the dynamical curvature is parallel in the sense of Proposition \ref{pps_equivariance_master}, then such a conformal structure exists in the two following contexts:
\begin{itemize}
\item When the structure group $\SO(n-1)$ of the frame bundle $FM \to SM$ is simple ($n \neq 3, 5$) and the frame flow is ergodic;
\item When the dimension $n$ of $M$ is odd.
\end{itemize}

\subsection{Preparatory phase} \label{sec:prep}

The proof of Theorem \ref{thm_conj_Kanai} relies on the following reduction:
\begin{lem} \label{lem_reduction}
Let $M$ be a strictly $1/4$-pinched, negatively curved manifold. Fix a base point $v \in SM$ and denote 
$$H^{E_u} = \overline{\Hol}^{\dyn}(E_u, d\phi_t, v)$$
the full dynamical holonomy group at $v$ associated to the flow $d\phi_t$, constructed in Theorem \ref{thm_stable_hol}.

If $H^{E_u}$ preserves the conformal class $[h_v]$ of a certain Riemannian metric $h_v$ on the fiber $E_u \vert_v$, then $M$ is homothetic to a real hyperbolic manifold.
\end{lem}
\begin{proof}
We define a conformal structure on $E_u$ as follows. Given any $w \in SM$, there exists a flow-stable-unstable path $\gamma$ (see Lemma \ref{lem_exist_csu}) between $v$ and $w$; then the associated holonomy $\Pi_{E_u}(\gamma)$ maps $E_u \vert_v$ to $E_u \vert_w$. Denote $[h_w]$ the pullback of $[h_v]$ by $\Pi_{E_u}(\gamma)^{-1}$. As $H^{E_u}$ preserves $[h_v]$, it follows that the conformal class $[h_w]$ does not depend on the choice of $\gamma$.

As a result, we obtain a conformal structure on $E_u$ which is moreover invariant under dynamical holonomies by construction. In particular, it is invariant under the action of $d\phi_t$, which acts by dynamical holonomies by construction. Moreover, this conformal structure is continuous, as the dynamical holonomies are continuous by Theorem \ref{thm_stable_hol}.

Then, it follows from \cite[Theorem 1]{kanai1993differential} that the geodesic flow $\phi_t$ on $SM$ is homothetic to the geodesic flow $\phi_t^0$ on the unit tangent bundle $SM^0$ of a real hyperbolic manifold $M^0$. By the main theorem of \cite{besson1995entropies}, this conjugacy is induced by an homothety between $M$ and $M^0$; hence the result.
\end{proof}

We end this section with the following useful lemmas:
\begin{lem} \label{lem_norm_so}
Let $G$ be a connected, closed subgroup of $\SO(k)$ acting transitively on the sphere $S^{k-1} \subset \R^k$. Then, the normalizer of the Lie algebra $\frk{g}$ of $G$ is contained in the conformal group $\CO(k)$.
\end{lem}
\begin{proof}
Let $A$ be an element of $\GL_k(\R)$ normalizing $\frk{g}$. Then, for any $X \in \frk{g}$, we have $A X A^{-1} \in \frk{g}$. Taking the exponential, we see that for any $U \in G$, we have $A U A^{-1} \in G$. 

By assumption, given $x, y \in S^{k-1}$, there exists $U \in G$ such that $U(x) = y$. Let $V = A U A^{-1} \in G$, then $\vert Ay \vert = \vert AUx \vert = \vert VAx \vert = \vert Ax \vert$, which shows that $A$ sends the unit sphere to some other sphere $r S^{k-1}$. As a result, $r^{-1} A \in \Or(k)$, which shows that $A$ is a conformal transformation of $\R^k$.
\end{proof}

Incidentally, groups acting transitively on the sphere have been classified (see for instance \cite[Theorem 4.8]{yasukura1986orthogonal}).

\begin{lem} \label{lem_so_to_conf}
Let $H$ be a closed subgroup of $\GL_k(\R), k \geq 3$ such that $\SO(k) \subset H$ and $H$ acts by conjugacy on some subspace $E \subset \End(\R^k)$ of dimension $d$ with $1 < d < k^2 - 1$. Then $H$ is contained in the conformal group $\CO(k)$.
\end{lem}
\begin{proof}
Let $PH_0$ be the projection of the identity component $H_0$ in $\PGL(\R^k)$. By assumption, $\PSO(k) \subset PH_0$ and by connectedness, $PH_0 \subset \PSL(\R^k)$. As $\PSO(k)$ is a maximal connected Lie subgroup of $\PSL(\R^k)$, we have $PH_0 = \PSO(k)$ or $PH_0 = \PSL(\R^k)$.

We claim that $PH_0 = \PSO(k)$. Indeed, remark first that $PH_0$ acts on $E$. The action of $\PSL(\R^{k})$ on $\End(\R^{k})$ splits as 
$$\End(\R^{k}) = \R I_{k} \oplus \frk{sl}_{k}(\R)$$
where $\frk{sl}_{k}(\R)$ is the set of matrices with zero trace. By assumption, there is no $\PSL(\R^k)$-subrepresentation with the same dimension as $E$. This implies that $PH_0 \neq \PSL(\R^k)$. As a result, $PH_0 = \PSO(k)$.

Thus, $H_0 \subset \CO(k)$. It follows that the Lie algebra $\frk{h}$ of $H$ is either $\frk{so}(k)$ or $\frk{so}(k) \oplus \R I_k$. 

Then, $H$ acts by conjugacy on its Lie algebra $\frk{h}$, and it always acts by conjugacy on $\frk{sl}(\R^k)$. This implies that $H$ acts on $\frk{so}(k) = \frk{h} \cap \frk{sl}(\R^k)$. The corresponding connected Lie subgroup $\SO(k)$ acts transitively on the sphere $S^{k-1}$, hence by Lemma \ref{lem_norm_so}, $H$ is contained in the conformal group $\CO(k)$.
\end{proof}

\begin{lem} \label{lem_flat_bundle}
Let $M$ be a smooth manifold of dimension $n$. If there exists a flat principal connection on the principal bundle $FM \to SM$, then $n = 1, 2, 4, 8$.
\end{lem}
\begin{proof}
Let $x \in M$ be a base point. Then, the fiber $S_x M$ is identified to the sphere $S^{n-1}$. Under this identification, the restriction $FM \vert_{S_xM} \to S_x M$ is isomorphic to the frame bundle $FS^{n-1} \to S^{n-1}$. 

If there exists a flat connection $\theta$ on $FM \to SM$, then $\theta$ restricts to a flat connection on the frame bundle of the sphere $FS^{n-1}$. Equivalently, the sphere $S^{n-1}$ is parallelizable. This implies that $n = 1, 2, 4$ or $8$.
\end{proof}

\begin{lem} \label{lem_levi}
Let $\frk{h}$ be a Lie algebra and $\frk{a} \subset \frk{h}$ an ideal. If $\frk{s} \coloneq \frk{h} / \frk{a}$ is semisimple, then there exists an embedding of Lie algebras $\frk{s} \hookrightarrow \frk{h}$.
\end{lem}

\begin{proof}

Denote $\frk{r}$ the radical of $\frk{h}$, which is the maximal solvable ideal of $\frk{h}$. Remark that $\frk{r}$ is necessarily contained in $\frk{a}$. Indeed, the image $\frk{r} / (\frk{r} \cap \frk{a}) \subset \frk{h} / \frk{a} = \frk{s}$ is a solvable ideal of $\frk{s}$, hence is trivial as $\frk{s}$ is semisimple. 

Denote $\frk{l} = \frk{h} / \frk{r}$ the Levi subalgebra of $\frk{h}$. It follows that $\frk{s}$ is a quotient of $\frk{l}$. As $\frk{l}$ is semisimple by construction, one can decompose 
$$\frk{l} = \frk{l}_1 \oplus \cdots \oplus \frk{l}_k$$
where each $\frk{l}_i$ is a simple ideal of $\frk{l}$. Moreover, any ideal of $\frk{l}$ is a sum of such $\frk{l}_i$. This implies that any quotient of $\frk{l}$ is isomorphic to a sum of $\frk{l}_i$, hence embeds into $\frk{l}$. In particular, $\frk{s}$ embeds into $\frk{l}$.

Then, it follows from the celebrated Levi decomposition theorem \cite[Appendix E.1]{fulton2013representation} that $\frk{l}$ embeds into $\frk{h}$, hence the result.
\end{proof}

\subsection{Generic case} \label{sec:kanai_generic}

The goal of this section is to prove the following:
\begin{pps}[Theorem \ref{thm_conj_Kanai}, generic case] \label{pps_kanai_generic}
Let $M$ be a negatively curved, $\delta$-pinched manifold with $\delta > 1/4$, of dimension $n \neq 3, 5$. Assume that the frame flow $\Phi_t^{FM}$ is ergodic with respect to the natural measure on $FM$. 

If the dynamical horizontal distribution $\mathbb{H}$ is $\mathcal{C}^{1/\sqrt{\delta}}$, then:
\begin{itemize}
\item If $n \neq 4, 8$, then $M$ is homothetic to a real hyperbolic manifold.
\item If $n = 4, 8$, then either $M$ is homothetic to a real hyperbolic manifold or $\mathbb{H}$ is integrable.
\end{itemize}
\end{pps}

\begin{proof}
Fix a base point $v \in SM$. Denote $H^{FM} = \overline{\Hol}^{\dyn}(FM, \Phi_t^{FM}, v)$ and $H^{E_u} = \overline{\Hol}^{\dyn}(E_u, d\phi_t, v)$ the full dynamical holonomy groups associated to the frame flow $\Phi_t^{FM}$ and the geodesic flow $\phi_t$ respectively. Note that by Theorem \ref{thm_ergod_frame} and as we assumed $\Phi_t^{FM}$ to be ergodic, we have $H^{FM} = \SO(n-1)$ is equal to the structure group of $FM \to SM$.

By Lemma \ref{lem_reduction}, we only have to show that $H^{E_u}$ preserves a conformal structure on the fiber $E_u \vert_v$. To do this, we shall use the lemmas of Section \ref{sec:equivariance} in order to obtain informations on $H^{E_u}$ from the existence of the "equivariant" (in the sense of Corollary \ref{cor_eq_curvature}) morphism $F$. However, the fact that $F$ is not an isomorphism in general prevents us in particular from using Proposition \ref{pps_eq_orbits_alg}. Therefore, we have to restrict and co-restrict $F$ appropriately, which complicates the argument. 

\bigskip

\textbf{Step 1:} Study of the domain and codomain of $F$.

Let $V_1 = \Ad(FM)\vert_v, V_2 = \End(E_u)\vert_v, f = F_v$ the curvature morphism associated to the frame flow at $v$, $\mathcal{P} = \mathcal{P}_v$ the monoid of flow-stable-unstable loops at $v$, $\rho_1 = \Pi_{FM}, \rho_2 = \Pi_{E_u}$ the dynamical holonomies.

The structure group $G = \SO(n-1)$ of the frame bundle $FM \to SM$ is a simple Lie group. In particular, the adjoint representation $\ad: G \to \GL(\frk{g})$ of $G$ is irreducible. The image $G_1 \coloneq \ad(G) \subset \GL(\frk{g})$ is isomorphic to $\PSO(n-1)$. By ergodicity, we have seen that $H^{FM} = \SO(n-1)$. Thus, by definition, the closure $H_1$ of the image of $\rho_1$ in $G_1$ is equal to $G_1 \subset \GL(V_1)$.

The group $H^{E_u}$ acts by conjugacy on $\End(E_u)$; denote $PH^{E_u}$ the image of $H^{E_u}$ in $\GL(\End(E_u)) \eqcolon G_2$; this group can be identified with the image of $H^{E_u}$ by the projection $\GL(E_u) \to \PGL(E_u)$. By definition, $H_2 = PH^{E_u} \subset G_2$.

\smallskip

We now apply Proposition \ref{pps_eq_orbits_repr} in this context. By Corollary \ref{cor_eq_curvature}, $f$ is $(\rho_1, \rho_2)$-equivariant. Moreover, $V_1$ is isomorphic to the adjoint representation $\frk{g}$, hence is $H_1$-irreducible by simplicity of $\SO(n-1)$. Thus, by Proposition \ref{pps_eq_orbits_repr}, $f = F_v$ is either zero or injective. Moreover, $F$ has constant rank by Corollary \ref{cor_rank}.

If $F_v$ is zero, then by definition (and Frobenius's theorem) the dynamical horizontal distribution $\mathbb{H}$ is integrable. This can only happen when $n = 4$ or $n = 8$ by Lemma \ref{lem_flat_bundle}.

If $F$ is injective, then, by Proposition \ref{pps_eq_orbits_repr}, the range $E$ of $F$ is an irreducible $H^{E_u}$-subrepresentation of $\End(E_u)$.

\bigskip

\textbf{Step 2:} Applying the equivariance properties.

In the following, we may assume that $F$ is injective. Denote $\alpha: H^{E_u} \to \GL(E)$ the restriction of the action of $H^{E_u}$ to $E \subset \End(E_u)$.

We now repeat the previous argument, but this time with $F$ co-restricted to its image $E$. Let $V_1 = \Ad(FM)\vert_v, V_2 = E_v, f = F_v$, $\mathcal{P} = \mathcal{P}_{v, 0}$ the monoid of \emph{contractible} flow-stable-unstable loops at $v$, $G_1 = \ad(\SO(n-1)) \subset \GL(V_1), G_2 = \GL(V_2)$, $\rho_1 = \Pi_{FM}, \rho_2 = \alpha \circ \Pi_{E_u}$. Thus, $H_1 = \ad(\SO(n-1)) \subset \GL(V_1)$ and it follows from Corollary \ref{cor_restrict_hol} that $H_2$ is the \emph{identity component} of $\alpha(PH^{E_u})$. In particular, $H_2$ is \emph{connected}.

As $G_1 \cong \PSO(n-1)$ is compact and $f$ is bijective, Propositions \ref{pps_eq_orbits_top}, \ref{pps_eq_orbits_alg} show that there exists a $H_2$-equivariant injective morphism of Lie groups $\chi: H_2 \hookrightarrow H_1$. 

\smallskip

We claim that $\chi$ is an isomorphism. First, recall that $V_1$ is isomorphic to the adjoint representation $\frk{h}_1$ of $H_1$ by construction; in particular $V_1$ is $H_1$-irreducible as seen earlier. It follows from Proposition \ref{pps_eq_orbits_repr} that $V_1$ is $\chi(H_2)$-irreducible. 

However, $\chi(H_2)$ will preserve its own Lie algebra, seen as a Lie subalgebra of $\frk{h}_1$. By irreducibility, the Lie algebra of $\chi(H_2)$ (hence of $H_2$) is either isomorphic to $\frk{h}_1$ or zero. 

Then, $H_2$ is connected. Hence, it it were discrete, it would be trivial; this contradicts the fact that $H_2$ acts irreducibly on $V_1$, which has dimension $> 1$. Thus, the Lie algebra of $\chi(H_2)$ is equal to $\frk{h}_1$. By connectedness of $H_1 \cong \PSO(n-1)$, it follows that $\chi(H_2) = H_1$, as claimed.

\bigskip

\textbf{Step 3:} From a quotient to a subgroup.

We have seen that when $F$ is injective, then the identity component $H_2$ of $\alpha(PH^{E_u})$ is isomorphic to $\PSO(n-1)$. Note that $H_2$ is a quotient of the identity component $H_0^{E_u}$. In particular, as $n \neq 3$, $H_0^{E_u}$ admits a compact semisimple \emph{quotient} with Lie algebra isomorphic to $\frk{so}(n-1)$.

Let $\frk{h}^{E_u}$ be the Lie algebra of $H_0^{E_u}$ (or equivalently of $H^{E_u}$). It follows from Lemma \ref{lem_levi} that $\frk{so}(n-1)$ embeds into $\frk{h}^{E_u}$.

Let $L$ be the connected Lie subgroup of $H^{E_u}$ which corresponds to this subalgebra. Note that $L \subset \GL(E_u)$. Then, $L$ has $\frk{so}(n-1)$ as Lie algebra, hence is covered by $\Spin(n-1)$. In particular, $L$ is compact as $n \neq 3$. 

As a result, there exists an inner product $h_v$ on $E_u \vert_v$ which is preserved by $L$. Equivalently, $L \subset \SO(h_v)$ (which is the associated orthogonal group), as $L$ is connected. However, $L$ has dimension $\dim \frk{so}(n-1) = (n-1)(n-2)/2$ which is the same as $\dim \SO(h_v)$, hence $L = \SO(h_v)$. As a result, $\SO(h_v) = L \subset H_0^{E_u} \subset H^{E_u}$.

\bigskip

\textbf{Step 4:} End of the proof.

Recall that $H^{E_u}$ acts on the image $E$ of $F_v$. We have seen that $E$ is isomorphic through $F_v$ to the adjoint representation $\frk{so}(n-1)$ of $\SO(n-1)$, in particular $E$ has dimension $(n-1)(n-2)/2$. 

We just proved that if $F$ is injective, then $H^{E_u}$ admits as subgroup $\SO(h_v)$ for some Riemannian metric $h_v$ on $E_u \vert_v$. Thus, we may apply Lemma \ref{lem_so_to_conf} to obtain that $H^{E_u}$ is contained in the conformal group $\CO(h_v)$. It only remains to apply Lemma \ref{lem_reduction} to obtain that $M$ is homothetic to a real hyperbolic manifold.

On the other hand, if $F$ is not injective, then we have proved in \textbf{Step 1} that the dynamical horizontal distribution is integrable and $n \in \{4, 8\}$, which ends the proof.
\end{proof}

\subsection{Odd dimensional case} \label{sec:kanai_odd}

The goal of this section is to prove the following:
\begin{pps}[Theorem \ref{thm_conj_Kanai}, odd dimensional case] \label{pps_Kanai_odd}
Let $M$ be a negatively curved, $\delta$-pinched $n$-manifold with $\delta > 1/4$. Assume that $n \geq 5$ is odd.

If the dynamical horizontal distribution $\mathbb{H}$ is $\mathcal{C}^{1/\sqrt{\delta}}$, then $M$ is homothetic to a real hyperbolic manifold.
\end{pps}

\smallskip

The proof of this result hinges on the topology of principal bundles over even dimensional spheres. In particular, we will be interested in the holonomy bundle associated to the differential $d\phi_t$ of the geodesic flow, which we constructed in Proposition \ref{pps_hol_bdle}.

For technical reasons, we consider the universal cover $\widetilde{M}$ of $M$; note that the unit tangent bundle $S\widetilde{M}$ may also be identified with $\widetilde{SM}$ as the spheres $S^{n-1}$ are simply connected for $n \geq 3$. One may then lift all the bundles and holonomies considered so far. Let us fix a point on $\tilde{x} \in \widetilde{M}$ and consider $S_{\tilde{x}} \widetilde{M} \cong S^{n-1}$. Also denote $\tilde{v} \in S^{n-1} \cong S_{\tilde{x}}M$ the North pole.

Denote $H^{\tilde{E}_u} = \overline{\Hol}^{\dyn}(\widetilde{E}_u, d\phi_t, \tilde{v})$ the full holonomy group at $\tilde{v}$ associated to the differential of the geodesic flow on $S\widetilde{M}$. Also denote $F_{\GL} \widetilde{E_u}$ the frame bundle of $\widetilde{E}_u$; this is a principal $\GL_{n-1}(\R)$-bundle. By Proposition \ref{pps_hol_bdle}, there exists a reduction $Q$ of the structure group $\GL_{n-1}(\R)$ of $F_{\GL} \widetilde{E}_u$ to $H^{\widetilde{E}_u}$.

Recall that $\widetilde{E}_u \vert_{S_{\tilde{x}}\widetilde{M}}$ is topologically isomorphic to the tangent bundle $TS^{n-1}$ over $S^{n-1}$ (see \S\eqref{eq:klingenberg}). This implies that $F_{\GL} \widetilde{E}_u \vert_{S_{\tilde{x}} \widetilde{M}}$ is topologically isomorphic to the frame bundle $F_{\GL} S^{n-1}$ of the sphere $S^{n-1}$. As a result, restricting $Q$ to $S_x M$ defines a reduction of the structure group of the frame bundle $F_{\GL} S^{n-1}$ of the sphere to $H^{\widetilde{E}_u}$. This, in turn, implies the following:

\begin{lem} \label{lem_red_K}
Let $K$ be a maximal compact subgroup of $H^{\widetilde{E}_u}$. There is a reduction of the structure group of $S^{n-1}$ to $K$.
\end{lem}
\begin{proof}
A general argument (see \cite[Vol.1, Ch.1, Example 5.6]{kobayashi1996foundations}) shows that as $H^{\widetilde{E}_u}$ is connected (by Corollary \ref{cor_restrict_hol}, this is why we consider the universal cover), the holonomy bundle $Q\vert_{S_{\tilde{x}}\widetilde{M}} \subset F_{\GL} S^{n-1}$ has a reduction to any maximal compact subgroup $K$ of $H^{\widetilde{E}_u}$. Restricting this reduced bundle to a fiber $S_{\tilde{x}} \widetilde{M}$ as above, we obtain the desired reduction of the structure group of $S^{n-1}$ to $K$.
\end{proof}

Doing the previous construction on the base manifold $SM$ at the projection $(x, v)$, the Corollary \ref{cor_restrict_hol} lets us then identify $H^{\tilde{E}_u} = H^{E_u}_0$ where $H^{E_u}_0$ is the connected component of the identity of the holonomy group $H^{E_u} = \overline{\Hol}^{\dyn}(E_u, d\phi_t, v) \subset \GL(E_u \vert_v) \cong \GL(\R^{n-1})$. As a result, we may consider that $K \subset H^{E_u}_0 = H^{\tilde{E}_u} \subset H^{E_u} \subset \GL(\R^{n-1})$ in the following.

\smallskip

We can now give the proof of Proposition \ref{pps_Kanai_odd}.
\begin{proof}
We first assume that $n \neq 7$. By Lemma \ref{lem_red_K}, there exists a reduction of the structure group of the sphere $S^{n-1}$ to $K$. By Theorem \ref{thm_leo_odd}, as $n \neq 7$, this implies that $K$ is the orthogonal group $\SO(h_v)$ associated to some Riemannian metric $h_v$ on $E_u \vert_v$. 

Recall that as $n \neq 7$, $\Phi_t$ is ergodic by \cite{brin1982ergodic}. We argue as in \textbf{Step 1} of the proof of Proposition \ref{pps_kanai_generic} and obtain that $H^{E_u}$ acts (by conjugacy) on a subspace $E \subset \End(E_u \vert_v)$, obtained as the range of $F_v$ where $F$ is the curvature morphism associated to the frame flow. Thus $E$ is nonzero (as $n \neq 4, 8$) and $\ker F_v$ is a subrepresentation of the adjoint representation $\frk{so}(n-1)$ of $\SO(n-1)$.

We can use this to compute the dimension of $E$:
\begin{itemize}
\item If $n \geq 9$, then $\frk{so}(n-1)$ is irreducible; this implies that $\ker F_v$ is zero. Thus, $\dim E = \dim \frk{so}(n-1) = (n-1)(n-2)/2$.
\item If $n = 5$, then $\frk{so}(4)$ splits as $\frk{so}(3) \oplus \frk{so}(3)$, hence $\dim \ker F_v \in \{0, 3\}$ and $\dim E \in \{3, 6\}$.
\end{itemize} 
Thus, by Lemma \ref{lem_so_to_conf}, $H^{E_u} \subset \CO(h_v)$. It only remains to apply Lemma \ref{lem_reduction} to obtain that $M$ is homothetic to a real hyperbolic manifold.

\bigskip 

We now consider the case $n = 7$. This time, the frame flow $\Phi_t^{FM}$ is not known to be ergodic in general. However, the associated full holonomy group $H^{FM} = \overline{\Hol}^{\dyn}(FM, \Phi_t^{FM}, v)$ is a closed subgroup of $\SO(6)$ and there is a reduction of the structure group of $S^6$ to $H^{FM}$, which may be obtained by restricting the associated holonomy bundle (see Proposition \ref{pps_hol_bdle}) to a fiber $S_x M \cong S^6$. By Theorem \ref{thm_leo_odd}, this implies that $H^{FM} = \SO(6), \U(3)$ or $\SU(3)$. If $H^{FM} = \SO(6)$, then $\Phi_t^{FM}$ is ergodic by Theorem \ref{thm_ergod_frame}, and we may conclude as above.

\smallskip

If $H^{FM} = \SU(3)$, we have to go back to the proof of Proposition \ref{pps_kanai_generic}. The only place where we used the ergodicity of $\Phi_t^{FM}$ was in order to show that $H_1 = \ad(\SO(n-1))$. This time, we will have $H_1 = \ad(\SU(3))$. However, $\ad(\SU(3))$ is still a simple Lie group and its adjoint representation is irreducible. Thus, we may repeat \textbf{Step 1} and \textbf{Step 2} of the proof of Proposition \ref{pps_kanai_generic} word for word, replacing $\PSO(n-1)$ with $\PSU(3)$. This shows that the holonomy Lie algebra $\frk{h}^{E_u}$ admits a quotient isomorphic to $\frk{su}(3)$. 

Then, $H^{E_u}$ also admits a subgroup conjugate to $\SU(3)$ by Lemma \ref{lem_red_K}. We may then list the subalgebras of $\End(\R^6)$ which contain $\frk{su}(3)$: they are $\{0\}, \frk{su}(3), \frk{u}(3), \frk{sp}(6, \R), \frk{sl}(\C^3), \frk{so}(6), \frk{sl}(\R^6)$ as well as each sum of $\R I_6$ and one of the previous algebras. This implies that $\frk{h}^{E_u}$ is isomorphic to $\frk{su}(3), \frk{u}(3), \frk{su}(3) \oplus \R I_6, \frk{u}(3) \oplus \R I_6$. 

Then, $H^{E_u}$ acts on its own Lie algebra and on $\frk{sl}(\R^6)$ by conjugacy, hence on their intersection. As a result, $H^{E_u}$ preserves either $\frk{su}(3)$ or $\frk{u}(3)$. The correpsonding linear Lie groups $\SU(3), \U(3)$ act transitively on the sphere $S^5 \subset \R^6$. Thus, in both cases, we may conclude that $H^{E_u}$ is contained in the conformal group of some Riemannian metric on $E_u \vert_v$ by Lemma \ref{lem_so_to_conf}. It only remains to apply Lemma \ref{lem_reduction} to obtain that $M$ is homothetic to a real hyperbolic manifold.

\smallskip 

If $H^{FM} = \U(3)$, then by the same argument as above, $H^{E_u}$ preserves a nonzero subspace $E \subset \End(E_u)\vert_v$ isomorphic to an ideal of $\frk{u}(3) \cong \R \oplus \frk{su}(3)$. If $\dim E \neq 1$, then we may reduce to the case of $\SU(3)$ above by considering a subspace of $E$ corresponding to $\frk{su}(3)$.

Else, $E$ is 1-dimensional. This implies that the curvature $\omega$ associated to the dynamical connection form $\Theta^{FM}$ is valued in a real line bundle; further, by Corollary \ref{cor_ambrose_singer}, the holonomy Lie algebra of the dynamical connection form has dimension $\leq 1$.

By lifting $\Theta$ to the universal cover $F \widetilde{M}$ and Corollary \ref{cor_restrict_hol}, we obtain that the bundle $F\widetilde{M} \to S \widetilde{M}$ has a reduction to a subgroup of dimension $\leq 1$. Restricting this bundle to a single sphere as before, we obtain that the frame bundle over $S^6$ has a reduction to a subgroup of dimension $\leq 1$. By Theorem \ref{thm_leo_odd}, this is impossible. This concludes the proof.
\end{proof}

\section{Proofs of the two generalizations} \label{sec:generalizations}

In this section, we generalize the proofs of the previous section in three contexts, proving Theorems \ref{thm_conj_Kanai_simple} and \ref{thm_conj_Kanai_odd}. The goal is again to prove the existence of a conformal structure on $E_u$ invariant under the geodesic flow (and then use Lemma \ref{lem_reduction}). 

This is done by first generalizing the tools in the proof of Proposition \ref{pps_kanai_generic}; this yields Proposition \ref{pps_synth_eq}. We obtain as a direct corollary Theorem \ref{thm_conj_Kanai_simple}, when the structure group is simple. In the odd dimensional case, the proof of Theorem \ref{thm_conj_Kanai_odd} is however more involved than Proposition \ref{pps_Kanai_odd}.

The last paragraph is dedicated to the study of extensions of general Anosov contact flows and is largely independent from the others.

\subsection{The equivariant correspondence} \label{sec:eq_correspondence}

Let $(P, \Phi^P_t)$ be an isometric extension of an arbitrary Anosov flow $\varphi_t$ on a closed manifold $N$, and assume that $\varphi_t$ is $\beta$-bunched for some $\beta \leq 1$. Assume that the dynamical connection form associated to $\Phi^P_t$ is $\mathcal{C}^{1, \beta}$. 

Fix $v \in N$ a base point. We denote by $\frk{d}^P$ the Lie algebra of the \emph{classical} holonomy group $D^P \coloneq \Hol^{\dyn}(P, \Phi_t^P, v)$ (Definition \ref{def:classical}), $\frk{h}^P$ the Lie algebra of the \emph{full} holonomy group $H^P \coloneq \overline{\Hol}^{\dyn}(P, \Phi_t^P, v)$ (Definition \ref{def:full}). By Corollary \ref{cor_inclusions}, we have a dense inclusion $D^P \subset H^P$.

We also denote by $\frk{h}^{E_u}$ the Lie algebra of the full holonomy group $H^{E_u} = \overline{\Hol}^{\dyn}(E_u, d\phi_t, v)$.

By Corollary \ref{cor_ambrose_singer}, one can identify $\frk{d}^P = (\ker F_v)^\perp$ where $F$ is the curvature morphism associated to $(P, \Phi_t^P)$ defined in \eqref{eq:identifications}. We also denote $E = F(\frk{d}^P) \subset \End(E_u)$.

We have:
\begin{lem} \label{lem_dom_codom}
The adjoint action of $\frk{h}^P$ preserves $\frk{d}^P$, the action by conjugacy of $\frk{h}^{E_u}$ preserves $E$, and $F_v$ restricts to an isomorphism between $\frk{d}^{P}$ and $E$.
\end{lem}
\begin{proof}
The group $D^P$ is normal in its closure $H^P$, hence its Lie algebra $\frk{d}^P$ is an ideal of $\frk{h}^P$; equivalently, $\frk{d}^P$ is a subrepresentation of the adjoint representation $\frk{h}^P$.

It follows from Corollary \ref{cor_ambrose_singer} that $F_v: \frk{d}^P \to E$ is an isomorphism.

Finally, the fact that $E$ is stable under $\frk{h}^{E_u}$ follows from Proposition \ref{pps_eq_orbits_repr} applied in the following context. Take $V_1 = \Ad(P)_v, V_2 = \End(E_u)\vert_v, f = F_v, G_1 = G, G_2 = \GL(\End(E_u))$ and $\rho_1 = \Pi_P, \rho_2 = \Pi_{E_u}$. Then the action of $H^{E_u}$ on $\End(E_u)$ induces the action of the corresponding $H_2$ on $\End(E_u)$, which preserves $E$ by Proposition \ref{pps_eq_orbits_repr}.
\end{proof}

If $\frk{h}$ is a Lie algebra acting on a vector space $V$ through a representation $\rho: \frk{h} \to \End(V)$ and $B \subset \End(V)$, we denote $C_{\frk{h}}(B)$ the \emph{centralizer} of $B$ in $\frk{h}$, defined by 
$$C_{\frk{h}}(B) = \{X \in \frk{h}, \forall Y \in B, [\rho(X), Y] = 0\}.$$
Here, we consider more specifically $\frk{h}^P$ acting on itself via the adjoint representation and $\frk{h}^{E_u}$ acting on $\End(E_u)$, where the action is induced by the associated Lie group action of $H^{E_u}$ on $E_u$.

The goal of this paragraph is to prove the following.

\begin{pps}
\label{pps_synth_eq}
Let $(P, \Phi_t^P)$ be an isometric extension of a $\beta$ bunched contact Anosov flow $\varphi_t$, $\beta \leq 1$. Assume that the dynamical connection form associated to $\Phi^P_t$ is $\mathcal{C}^{1, \beta}$.

There is an isomorphism of Lie algebras
$$\frk{h}^P / C_{\frk{h}^P}(\frk{d}^P) \cong \frk{h}^{E_u} / C_{\frk{h}^{E_u}}(E) \eqcolon \frk{l}.$$
 
The Lie algebra $\frk{l}$ is compact and semi-simple. Moreover, $E$ is isomorphic as $\frk{l}$-representation to $\ad \frk{l} \oplus \R^d$, where $d$ is the dimension of the center of $\frk{d}^P$.
\end{pps}

\begin{proof}
\textbf{Step 1:} Restriction and co-restriction of $F$.

Denote $\alpha: \End(\frk{h}^P) \to \End(\frk{d}^P)$ and $\eta: \End(\End(E_u)) \to \End(E)$ the restriction maps. 

We consider $V_1 = \frk{d}^P, V_2 = E, f = F_v, G_1 = \SO(V_1), G_2 = \GL(V_2)$ and $\rho_1 = \Pi_P \vert_{V_1}, \rho_2 = \Pi_{E_u} \vert_{V_2}$ the restricted dynamical holonomies. We also consider $\mathcal{P} = \mathcal{P}_{v, 0}$ the monoid of \emph{contractible} flow-stable-unstable paths. 

The closures $H_1, H_2$ of the images of $\rho_1, \rho_2$ are computed as follows. First, the closure of the group generated by the image of $\Pi_P$ in $\frk{h}_P$ is by definition $H^P$, and the subgroup generated by dynamical holonomies is the identity component $H_0^P$ by Corollary \ref{cor_restrict_hol}. Thus, $H_1 = \alpha(H_0^P)$. Similarly, $H_2 = \eta(H_0^{E_u})$.

We can then rewrite 
$$C_{\frk{h}^P}(\frk{d}^P) = \ker \alpha \vert_{\frk{h}^P}, \quad C_{\frk{h}^{E_u}}(E) = \ker \eta \vert_{\frk{h}^{E_u}}.$$
and as a result, denoting $\frk{h}_i$ the Lie algebra of $H_i, i = 1, 2$, we obtain
$$\frk{h}^P / C_{\frk{h}^P}(\frk{d}^P) \cong \frk{h}_1, \quad \frk{h}^{E_u} / C_{\frk{h}^{E_u}}(E) \cong \frk{h}_2.$$

\bigskip

\textbf{Step 2:} The group $H_2$ is compact.

We claim that $H_2$ is compact. By Remark \ref{rmq_compact_H}, as $f$ is bijective and $G_1$ is compact, the $H_2$ orbits of $V_2$ are compact. Choose a basis $(e_i)$ of $V_2$, then if $h_j$ is a sequence of elements of $H_2$, one can extract a subsequence such that each $h_{j_k}(e_i), h_{j_k}^{-1}(e_i)$ is convergent. It follows that the subsequences $h_{j_k}, h_{j_k}^{-1}$ are convergent. If we denote $h, h'$ the respective limits of these subsequences, then $hh' = 1$, hence $h$ is invertible, and $h \in H_2$ as $H_2$ is closed. This proves the claim.

\bigskip

\textbf{Step 3:} The groups $H_1$ and $H_2$ are isomorphic, and the representations $\frk{d}^P, E$ are isomorphic.

As $f$ is bijective and $H_1$ is compact, we can apply Propositions \ref{pps_eq_orbits_top} and \ref{pps_eq_orbits_alg} in this context. We obtain an injective morphism of Lie groups $\chi: H_2 \hookrightarrow H_1$, which is the conjugacy by $f$. However, notice that as $H_2$ is compact too one may apply the same results with the roles of $H_1$ and $H_2$ exchanged, which shows that $\chi$ is an isomorphism.

The same results show that $f$ is an isomorphism of representations between $V_1 = \frk{d}^P$ and $V_2 = E$. More precisely, if we identify $H_1$ and $H_2$ through $\chi$, then $f$ becomes genuinely $H_1$ equivariant.

In particular, the Lie algebras $\frk{h}_1, \frk{h}_2$ are isomorphic (to the Lie algebra denoted $\frk{l}$ in the statement).

\bigskip

\textbf{Step 4:} The groups $H_1, H_2$ are semi-simple.

It suffices to check that the Lie algebra $\frk{h}_1$ is semisimple. As $H^P$ is a compact Lie group, one can decompose 
$$\frk{h}^P = \frk{s}_1 \oplus \cdots \oplus \frk{s}_k \oplus \frk{z}$$
where the $\frk{s}_i$ are simple commuting ideals and $\frk{z}$ is the center of $\frk{h}^P$.

As $\frk{d}^P$ is a subrepresentation of $\frk{h}^P$, it is an ideal and thus writes (up to some relabelling of the simple blocks)
$$\frk{d}^P = \frk{s}_1 \oplus \cdots \oplus \frk{s}_j \oplus \frk{z}'$$
where $\frk{z}'$ is the center of $\frk{d}^P$. In particular, one readily computes $C_{\frk{h}^P}(\frk{d}^P)$ and then $\frk{h}_1$ to be 
\begin{align*}
C_{\frk{h}^P}(\frk{d}^P) &= \frk{s}_{j + 1} \oplus \cdots \oplus \frk{s}_k \oplus \frk{z}, \\
\frk{h}_1 &= \frk{s}_1 \oplus \cdots \oplus \frk{s}_j.
\end{align*}
As a result, $\frk{h}_1$ is semi-simple.

\bigskip

\textbf{Step 5:} Information obtained on the representation $E$.

We have seen at the end of \textbf{Step 3} that $f$ is a morphism of representations between $V_1 = \frk{d}^P$ and $V_2 = E$. We have just computed that
$$\frk{d}^P \cong \frk{h}_1 \oplus \frk{z}'$$
where $\frk{z}'$ is the center of $\frk{d}^P$. Thus, as $\frk{h}_1$-representations, we have 
$$E \cong \frk{h}_1 \oplus \R^d$$
where $\R$ is the trivial representation and $d = \dim \frk{z}'$.
\end{proof}

\subsection{Isometric extensions of geodesic flows - Simple ergodic case} \label{sec:frame_generic}

In this section, we shall prove the following result, which refines Theorem \ref{thm_conj_Kanai_simple}:
\begin{thm}
\label{thm_kanai_gen}
Let $(M, g)$ be a $\delta$-pinched negatively curved manifold of dimension $n$ with $\delta > 1/4$ and let $(P, \Phi^P_t)$ be a $G$-principal isometric extension of the geodesic flow $\phi_t$ on $SM$ where $G$ is compact.

Assume that $G$ is a simple Lie group of dimension at least $(n-1)(n-2)/2$ and $\Phi_t^P$ is ergodic. If the dynamical horizontal distribution $\mathbb{H}$ is $\mathcal{C}^{1/\sqrt{\delta}}$, then, there are two possibilities:
\begin{itemize}
\item The dynamical horizontal distribution is integrable. Equivalently, $\Phi_t^P$ is a flat flow (see Section \ref{sec:flat}).
\item $M$ is homothetic to a real hyperbolic manifold, and $G$ has Lie algebra $\frk{so}(n-1)$. Further, there exists an isomorphism of Lie algebra bundles between the adjoint bundles $\psi: \Ad(P) \to \Ad(FM)$ conjugating the induced flows.
\end{itemize}
\end{thm}

\begin{rmq}
We believe that the isomorphism between adjoint bundles can be upgraded to some extent to a bundle isomorphism between $P$ and $FM$; certainly, such an isomorphism can only exist at the level of the universal covers. However, we were unable to prove this.
\end{rmq}

\begin{proof}
By Theorem \ref{thm_ergod_frame}, the full holonomy group $H^P$ of $\Phi_t^P$ is exactly $G$. Thus, the Lie algebra $\frk{h}^P$ of $P$ is equal to $\frk{g}$, hence is simple. As a result, $\frk{h}^P$ admits exactly two quotients, which are the trivial Lie algebra and $\frk{h}^P$ itself.

Let $\frk{d}^P$ be the Lie algebra of the \emph{classical} dynamical holonomy group $D^P$ associated to $\Phi_t^P$. We have by Proposition \ref{pps_synth_eq}
$$\frk{l} \coloneq \frk{h}^P / C_{\frk{h}^P}(\frk{d}^P) \cong \frk{h}^{E_u} / C_{\frk{h}^{E_u}}(E)$$
where $E$ is the range of the dynamical curvature map $F^P_v$ at $v$ and $\frk{h}^{E_u}$ is the Lie algebra of the full holonomy group $H^{E_u}$ associated to $d\phi_t$ on $E_u$.Thus, $\frk{l}$ is either trivial or isomorphic to $\frk{g}$. Recall that by Corollary \ref{cor_ambrose_singer}, $\frk{d}^P$ is complement to the kernel of the dynamical curvature $F_v^P$. Hence, either $F_v^P = 0$ or $F_v^P$ is injective. By Corollary \ref{cor_rank}, this implies that either $F^P = 0$ or $F^P$ is fiberwise injective.

If $F^P = 0$, then the dynamical connection associated to $\Phi_t^P$ is flat by definition. Equivalently, the dynamical horizontal distribution is integrable.

Else, $F^P$ is injective. Moreover, Lemma \ref{lem_levi} provides an embedding $\frk{g} = \frk{l} \hookrightarrow \frk{h}^{E_u}$. Let $L$ the connected subgroup of $H^{E_u}$ with Lie algebra $\frk{l}$. The universal cover of $L$ is compact by the theory of simple compact Lie algebras; hence $L$ is compact. As a result, $L \subset \End(E_u)$ preserves some inner product $h_v$ on $E_u \vert_v$, which provides an embedding $L \hookrightarrow \SO(h_v)$. However, $L$ has dimension at least $(n-1)(n-2)/2 = \dim \SO(h_v)$, hence $L = \SO(h_v)$ by connectedness. In particular, $\frk{g} \cong \frk{so}(n-1)$. 

Thus, we have shown that $\SO(n-1) \subset H^{E_u}$. Recall that by Lemma \ref{lem_dom_codom}, $H^{E_u}$ acts on the subspace $E \subset \End(E_u \vert_v)$ which has dimension $(n-1)(n-2)/2 = \dim \frk{so}(n-1) = \dim \frk{g}$ by injectivity of $F_v ^P$. The Lemma \ref{lem_so_to_conf} then shows that $H^{E_u}$ is contained in the conformal group $\CO(h_v)$. Thus, by Lemma \ref{lem_reduction}, $M$ is homothetic to a real hyperbolic manifold.

\smallskip

It remains to construct the conjugacy between adjoint bundles. By Proposition \ref{pps_equivariance_master}, the curvature $F^P$ conjugates $\Ad(\Phi_t^P)^*$ to $\Ad(d\phi_t)$. The same holds for $FM$ and the frame flow, thus $\psi \coloneq (F^{FM})^{-1} \circ F^P$ gives us our conjugacy. Moreover, $\psi$ is an isomorphism of $\frk{so}(n-1)$-representations by composition, hence it is a Lie algebra isomorphism as the representation here is the adjoint representation of $\frk{so}(n-1)$.
\end{proof}

\begin{rmq}
The dimension bound $\dim G \geq (n-1)(n-2)/2$ can actually be relaxed (with the same proof) to $\dim G > m(n-1)$, where $m(n)$ denotes the maximal dimension of a proper Lie subalgebra of $\frk{so}(n)$. This value can be determined directly, or by consulting the classification of maximal subgroups of $\frk{so}(n)$ (see \cite{dynkin1957maximal}): for $n \neq 4$, $m(n) = (n-1)(n-2)/2 = \dim \SO(n-1)$, while for $n = 4$, $m(4) = 4 = \dim \U(2)$.
\end{rmq}

\subsection{Isometric extensions of geodesic flows - Odd dimensional case} \label{sec:frame_odd}

We now give the proof of Theorem \ref{thm_conj_Kanai_odd}.

\begin{proof}
We use the same notations as in Section \ref{sec:eq_correspondence}.

By Lemma \ref{lem_red_K}, the maximal compact subgroup $K$ of the full holonomy group $H^{E_u}$ defines a $K$-structure on the sphere $S^{n-1}$. By Theorem \ref{thm_leo_odd}, this implies that $K$ is either $\SO(n-1)$ or maybe $\SU(3)$ or $\U(3)$ if $n = 7$. More precisely, there exists a metric $h_v$ on $E_u \vert_v$ such that $K$ is given by the standard embedding into $\SO(n-1) \cong \SO(h_v)$.

The strategy here is to use Proposition \ref{pps_synth_eq} in order to compute explicitly $\frk{h}^{E_u}$, knowing that it already contains the Lie algebra of $K$ and admits as quotient $\frk{l}$. We distinguish two cases.

\smallskip
\textbf{The Lie algebra $\frk{l}$ is trivial.} 

In this case, Proposition \ref{pps_synth_eq} yields that $\frk{d}^P$ is central in $\frk{h}^P$ and that $\frk{h}^{E_u}$ commutes with every element of $E$.

From this, we deduce that the Lie algebra of $K$ acts trivially on $E \subset \GL(E_u \vert_v) \cong \GL(\R^{n-1})$. The spaces of fixed points of $K$ are computed as follows:
\begin{itemize}
\item If $K = \SO(n-1)$ and $n \geq 4$, then the fixed points of $K$ are the homotheties $\R I_{n-1}$;
\item If $n = 3$, then the fixed points of $K$ are the homotheties $\R I_2$ and the skew-symmetric matrices $\frk{so}(2)$;
\item If $n = 7$ and $K = \SU(3)$ or $\U(3)$, then the fixed points of $K$ are the homotheties $\R I_6$ and the line $\R J$ where $J$ is the standard complex structure on $\R^6$.
\end{itemize}
This gives us every possibility for $E$. Using the fact that $H^{E_u}$ preserves $E$, we deduce that:
\begin{itemize}
\item If $E \subset \R I_{n-1}$, then $\Phi_t^P$ is a homothetical curvature flow by definition (Section \ref{sec:exp_flow}). 
\item If $n = 3$ and $E$ has a nontrivial $\frk{so}(2)$ component, then one checks directly that $H^{E_u}$ acts on $\frk{so}(2)$ (as it commutes with the homotheties $\R I_2$) and thus is a subset of $\CO(n-1)$ by Lemma \ref{lem_norm_so} for $G = \SO(n-1)$. This implies that $M$ is homothetic to a real hyperbolic manifold by Lemma \ref{lem_reduction}.
\item If $n = 7$ and $E$ has a nonzero $\R J$ component, then $H^{E_u}$ is contained in $\GL_3(\C)$. Arguing as in Lemma \ref{lem_reduction}, we may use parallel transport to obtain an holonomy invariant complex structure on $E_u$. However, such a structure does not exist as explained in the proof of \cite[Corollary 2.12]{hamenstadt1995invariant}. 
\end{itemize}
\smallskip

\textbf{The Lie algebra $\frk{l}$ is nontrivial.}

Assume that $\frk{l}$ is nontrivial. We first compute the projection $PH_0^{E_u} \subset \PGL(E_u)\vert_v$.
\begin{itemize}
\item Assume that $K = \SO(n-1)$. Consider $PH^{E_u}_0 \subset \PGL(E_u)\vert_v \cong \PGL(\R^{n-1})$; notice that $\PSO(n-1) \subset PH^{E_u}_0 \subset \PSL(\R^{n-1})$. By maximality of $\PSO(n-1)$ as connected subgroup of $\PSL(\R^{n-1})$, we obtain that $PH^{E_u}_0 = \PSO(n-1)$ or $\PGL(\R^{n-1})$. However, $\frk{l}$ is a compact, semisimple, nonzero quotient of the Lie algebra of $PH^{E_u}_0$ by construction, and $\PGL(\R^{n-1})$ is simple, noncompact hence does not have any compact nontrivial quotient. As a result, $PH^{E_u}_0 = \PSO(n-1)$.
\item Assume that $n = 7$ and $\SU(3) \subset K$. Similarly, one can list every closed, connected subgroup of $\PSL(\R^6)$ which contains $\PSU(3)$ \\ (equivalently, one can list the corresponding Lie algebras); these are $\PSU(3), P\U(3)$, $\PSO(6), P\Sp(6, \R), \PSL(\C^3), \PSL(\R^6)$. Here, $\Sp(6, \R)$ denotes the stabilizer of the standard nondegenerate 2-form on $\R^6$. Again, $\frk{l}$ is a compact, semisimple, nonzero quotient of the Lie algebra of $PH^{E_u}_0$, and $P\Sp(6, \R), \PSL(\C^3), \PSL(\R^6)$ do not have such a quotient (they do not have any simple compact ideal). As a result, $PH^{E_u}_0 = \PSU(3), P\U(3)$ or $\PSO(6)$.
\end{itemize}

We thus obtain that the identity component $H^{E_u}_0$ of $H^{E_u}$ is contained in $\CO(n-1)$. This implies that the Lie algebra $\frk{h}^{E_u}$ is contained in $\frk{so}(n-1) \oplus \R I_{n-1}$. Moreover, we know that $\frk{h}^{E_u}$ contains the Lie algebra of $K$, thus:
\begin{itemize}
\item If $K = \SO(n-1)$, then $\frk{h}^{E_u}$ is either $\frk{so}(n-1)$ or $\frk{so}(n-1) \oplus \R I_{n-1}$;
\item If $K = \SU(3)$ or $\U(3)$, then $\frk{h}^{E_u}$ is either $\frk{su}(3), \frk{u}(3), \frk{so}(6)$ or\\ $\frk{su}(3) \oplus \R I_6, \frk{u}(3) \oplus \R I_6, \frk{so}(6) \oplus \R I_6$.
\end{itemize}
Then, $H^{E_u}$ acts on its own Lie algebra $\frk{h}^{E_u}$ and also acts on $\frk{sl}(\R^{n-1})$, hence acts on $\frk{h}^{E_u} \cap \frk{sl}(\R^k)$ which is equal to $\frk{so}(n-1), \frk{su}(3)$ or $\frk{u}(3)$. As the groups $\SO(n-1), \SU(3)$ and $\U(3)$ act transitively on the corresponding unit sphere, we deduce from Lemma \ref{lem_norm_so} that $H^{E_u}$ is contained in $\CO(n-1)$.

We may then conclude from Lemma \ref{lem_reduction} that $M$ is homothetic to a real hyperbolic manifold.
\end{proof}

\begin{rmq}
It follows from Example \ref{ex:hyp_geod_group} that for real hyperbolic manifolds, the maximal compact subgroup $K$ is equal to $\SO(n-1)$ and the Lie algebra $\frk{l}$ is either $\frk{so}(n-1), n \neq 3$ or $0$ for $n = 3$. In particular, the case $n = 7$ is not special in our problem even though $S^6$ admits a reduction of its structure group to $\SU(3)$.
\end{rmq}

\subsection{Isometric extensions of contact Anosov flows}

In this section, we generalize Theorem \ref{thm_conj_Kanai_simple} to principal, isometric extensions $(P, \Phi_t)$ of contact Anosov flows $\varphi_t$. This is done by connecting the dynamical bundles associated to the extension $\Phi_t$ and the base flow $\varphi_t$, as described in the following:
\begin{pps}
\label{pps_high_regularity}
Let $\varphi_t$ be an Anosov, $\beta$ bunched flow on a closed $2d-1$ manifold $N$ with $\beta \leq 1$. Let $(P, \Phi_t^P)$ be an isometric principal extension of $\varphi_t$ with compact structure group. Assume that the dynamical horizontal space $\mathbb{H}$ associated to $\Phi^P_t$ is $\mathcal{C}^{k+1}$ for some $k \geq 2$. 

Then, if at some base point $v \in N$ the union of the ranges of elements in the range of $F_v$ in $\End(E_u \vert_v)$ generates $E_u \vert_v$, then the stable and unstable bundles associated to $\varphi_t$ are $\mathcal{C}^k$.
\end{pps}

\begin{proof}
Consider a piecewise smooth path $\gamma$ in $\mathcal{P}_v$. By the equivariance of $F$, we have
$$F_v \circ \Pi_{E_u}(\gamma) = \Pi_{P}(\gamma) \circ F_v.$$
By assumption, one may find $x_1, \cdots, x_{n-1} \in \Ad(P)_v, y_1, \cdots, y_{n-1} \in E_u \vert_v$ such that $y_i$ is a column of $F_v(x_i)$, say the $j_i$-th one, and $(y_i)$ form a basis of $E_u \vert_v$. It follows from the continuity of $F$ and of the determinant that for $t$ small enough, $y_i(t) = [F_v \circ \Pi_{P}(\gamma)(t)x_i]_{j_i}$ form a basis of $E_u \vert_{\varphi_t(v)}$.

By Lemma \ref{lem_exist_csu}, such (piecewise smooth) paths $\gamma$ let us build a local trivialization of $E_u$ around $v$ which has the same $\mathcal{C}^k$ regularity as $F \circ \Pi_P$. This proves that the unstable bundle is $\mathcal{C}^k$ in a neighborhood of $v$. This result then extends to the whole space by using the existence of parallel transport maps on $P$ between any pair of fibers (see Lemma \ref{lem_exist_csu}) which are $\mathcal{C}^k$ by assumption.
\end{proof}

As mentionned in the introduction, rigidity results for the stable and unstable bundles of geodesic flows were obtained in \cite[Theorem 1]{benoist1990flots} in high regularity. We now recall several notions from this paper. 

When classifying contact Anosov flows with smooth (un)stable bundles, one considers first the geodesic flows on locally symmetric manifolds; then one can also look at finite covers and/or quotients of these. This step is only necessary for locally symmetric 2-manifolds, as otherwise covers of geodesic flows on unit tangent bundles $SM$ are themselves geodesic flows (see Remark \ref{rmq_univ_cover}).

However, one may also consider smooth reparametrizations of the flow, which are described by a cohomology class of $H^1(N, \R)$ lying in a convex open neighborhood of the zero class. More precisely, given a closed 1-form $\alpha$ on our manifold $N$ such that $1 + \alpha(X) > 0$, where $X$ is the generator of $\varphi_t$, we may consider the generator $\hat{X}$ of the reparametrized flow $\hat{\varphi}_t$ given by the formula
$$\hat{X} = \frac{1}{1 + \alpha(X)} X.$$
As explained in \cite{benoist1990flots}, the flow $\hat{\varphi}_t$ only depends, up to smooth re-parametrization, on the cohomology class of $\alpha$ in $H^1(N, \R)$. Moreover, if $\lambda$ is the contact form associated to $\varphi_t$, then $\hat{\varphi}_t$ preserves the contact form $\lambda + \alpha$, and the associated stable and unstable bundles are smooth; they are given by the explicit formulas 
$$\hat{E}_{s/u} = \{ Y - \alpha(Y) X, Y \in E_{s/u} \}$$
where $E_{s/u}$ are the stable and unstable distributions of $\varphi_t$.

The result of \cite{benoist1990flots} states that we obtain this way every example of contact Anosov flow with smooth horospherical bundles, up to smooth conjugacy. Moreover, one can get rid of the conjugacy using \cite{besson1995entropies}. 

\begin{csq}
\label{cor_high_regularity}
Under the assumptions of Proposition \ref{pps_high_regularity}, if $d \geq 3$ and $k \geq 2(2d^2-d+1)$, then $(N, \varphi_t)$ can be obtained by deformation of a geodesic flow $\phi_t$ on a locally symmetric manifold $M$ by some cohomology class.
\end{csq}
\begin{proof}
This follows from \cite[Theorem 1]{benoist1990flots} and \cite{besson1995entropies}. Note that as $d \neq 2$, there is no need to consider covers or quotients of $(SM, \phi_t)$, as these are directly geodesic flows on locally symmetric manifolds (see Remark \ref{rmq_univ_cover}).
\end{proof}

We now state our result.

\begin{thm} \label{thm_general_simple}
Let $\varphi_t$ be an Anosov, $\beta$ bunched flow on a closed $2d-1$ manifold $N$ with $\beta \leq 1$. Let $(P, \Phi_t^P)$ be an isometric principal extension of $\varphi_t$ with compact structure group $G$. 

Assume that $\dim G \geq (d-1)(d-2)/2$, $\Phi_t^P$ is ergodic and that the dynamical horizontal space $\mathbb{H}$ associated to $\Phi^P_t$ is $\mathcal{C}^{k+1}$ for some $k \geq 2(2d^2-d+1)$. 

There are two possibilities:
\begin{itemize}
   \item[(1)] The dynamical curvature $F$ vanishes. In this case, $\Phi_t^P$ is a flat flow.
   \item[(2)] We can obtain $(N, \varphi_t)$ by deformation of a geodesic flow $\phi_t$ on a real hyperbolic manifold $M$ by some cohomology class. The corresponding reparametrization of $\Phi_t^P$ defines an isometric extension of $\phi_t$ on $SM$, which is then described by Theorem \ref{thm_conj_Kanai_odd}.
\end{itemize}
\end{thm}

\begin{proof}
The case $F = 0$ has already been studied in Section \ref{sec:flat}. We now assume that $F \neq 0$. Arguing exactly as in the proof of Theorem \ref{thm_kanai_gen}, we obtain that $F_v$ is injective at every point $v$, $\frk{g}$ is isomorphic to $\frk{so}(d)$ and $F_v$ can be identified with the standard embedding $\frk{so}(d) \subset \frk{gl}(\R^d) \cong \frk{gl}(E_u)\vert_v$. Moreover, the full holonomy group $H^{E_u}$ satisfies $\SO(d) \subset H \subset \CO(d)$.

We can then apply Corollary \ref{cor_high_regularity}. Indeed, the range of $F_v$ is precisely the subspace $\frk{so}(d) \subset \End(E_u)\vert_v$, and it is easy to check that any vector of $\R^d$ (say, any vector of the canonical basis) may be realized as a column of a matrix in $\frk{so}(d)$. We obtain that $\varphi_t$ is a deformation of a geodesic flow $\phi_t$ on a locally symmetric manifold $M$ by some cohomology class $\alpha$.

More explicitly, the generator $\hat{X}$ of $\hat{\phi}_t$ is given by the formula 
$$\hat{X} =\frac{1}{1 + \alpha(X)} X$$
where $X$ is the generator of $\phi_t$. If we still denote by $g$ the lift of the Riemannian metric $g$ on $M$ to $FM$, then the corresponding reparametrization of $\Phi^P_t$ will preserve the metric $\hat{g}$ on $FM$ given by the formula 
$$\hat{g}(Y, Z) = g(Y + \alpha(Y)X, Z + \alpha(Z)X)$$
where $Y, Z$ are vector fields on $M$. 

Thus, it only remains to study isometric extensions of geodesic flows on locally symmetric manifolds. Among them, only real hyperbolic manifolds admit a full holonomy group $H^{E_u}$ which can contain $\SO(d)$ as a subgroup. Indeed, the geodesic flow on the complex hyperbolic space preserves the complex structure, which prevents the full holonomy group $H^{E_u}$ from containing $\SO(d)$. A similar argument applies to the quaternionic hyperbolic space and octonionic plane.
\end{proof}

\printbibliography

\end{document}